\RequirePackage{fix-cm}
\documentclass[smallextended]{svjour3}       
\smartqed  
\usepackage{graphicx}
\usepackage{enumerate,bm,amsmath}
\usepackage{subfigure}
\usepackage{url}
\usepackage{amssymb}
\newcommand{\be}{\begin{equation}}
	\newcommand{\ee}{\end{equation}}
\numberwithin{equation}{section}
\newcommand{\lv}{\lvert}
\newcommand{\rv}{\rvert}
\newcommand{\pa}{\partial}

\newcommand{\la}{\langle}
\newcommand{\ra}{\rangle}
\newcommand{\ov}{\overline}

\begin{document}

\title{A semi-discrete first-order low regularity exponential integrator for the ``good" Boussinesq equation without loss of regularity
}
\subtitle{}

\titlerunning{FIRST-ORDER LREI FOR GB}        

	\author{Hang Li  \and Chunmei Su*
}


\institute{C. Su (*Corresponding author) \at
	Yau Mathematical Sciences Center, Tsinghua University, Beijing, China \\
	Tel.: +123-45-678910\\
	Fax: +123-45-678910\\
	\email{sucm@tsinghua.edu.cn}           
	\and
	H. Li \at
	Yau Mathematical Sciences Center, Tsinghua University, Beijing, China
}


\maketitle

\begin{abstract}
In this paper, we propose a semi-discrete first-order low regularity exponential-type integrator (LREI) for the ``good" Boussinesq equation. It is shown that the method is convergent linearly in the space $H^r$ for solutions belonging to $H^{r+p(r)}$ where $0\le p(r)\le 1$ is non-increasing with respect to $r$, which means less additional derivatives might be needed when the numerical solution is measured in a more regular space. Particularly, the LREI presents the first-order accuracy in $H^{r}$ with no assumptions of additional derivatives when $r>5/2$. This is the first time to propose a low regularity method which achieves the optimal first-order accuracy without loss of regularity for the GB equation. The convergence is confirmed by extensive numerical experiments.
\end{abstract}
\allowdisplaybreaks
\section{Introduction}
\label{intro}
We consider the following periodic boundary value problem of the ``good" Boussinesq (GB) equation:
\be\label{GB}
\left\{	\begin{aligned}
	&z_{tt}+z_{xxxx}-z_{xx}-(z^2)_{xx}=0, \quad x \in \ \mathbb{T}, \quad t>0,\\
	&z(0, x)=\phi_0(x), \quad z_{t}(0, x)=\psi_0(x),
\end{aligned}
\right.\ee
in a torus $\mathbb{T}=[-\pi,\pi]$, where $\phi_0(x)$ and $\psi_0(x)$ are given initial data. The GB equation was originally founded by Joseph
Boussinesq \cite{Boussinesq1872} to describe the propagation of dispersive shallow water waves. Furthermore, it was also extended by replacing the quadratic nonlinearity with a general function of $z$ to model small oscillations of nonlinear beams \cite{varlamov2001eigenfunction} or the two-way propagation of water waves in a channel. There have been many applications for the GB equation in physics \cite{johnson1997modern,kirby1996nonlinear} and oceanographic engineering \cite{tatlock2018assessment}.

Analytically, similar to the well-know Korteweg-de Vries (KdV) equation, the nonlinear Schr{\"o}dinger (NLS) equation, and other dispersive equations, the GB equation admits abundant soliton solutions, see \cite{el2003numerical,de1991pseudospectral,lambert1987soliton,manoranjan1984numerical,manoranjan1988soliton}. However, the GB equation has some special characteristics that make it different from the KdV or NLS equations, e.g., two solitons can merge into one soliton or develop into the so-called antisolitons \cite{manoranjan1988soliton}. For less smooth solutions, Kishimoto \cite{kishimoto2013sharp} gave a sharp locally well-posed result by using the fix-point theory together with low regularity bilinear estimates in Bourgain spaces, also known as the dispersive Sobolev spaces \cite{TTao}.
The main result in \cite{kishimoto2013sharp} is that for any $(\phi_0, \psi_0)\in H^{s}\times H^{s-2}$, $s\ge -1/2$, there exist a positive time $T(\|\psi_0\|_{H^{s}}, \|\psi_0\|_{H^{s-2}})>0$ and a unique solution of the GB equation in a certain Banach space of functions $X\subset C([0,T]; H^s\times H^{s-2})$, however, this equation is ill-posed when $s< -1/2$. We refer to \cite{farah2010periodic,farah2009local,kishimoto2013sharp,oh2013improved,wang2013well} for more detailed theoretical results of the GB equation.

Along the numerical part, a large variety of classical numerical schemes for approximating the time dynamics of the GB equation have been proposed and analyzed, including the pseudospectral methods \cite{cheng2015fourier,de1991pseudospectral}, finite difference methods \cite{bratsos2007second,ortega1990nonlinear}, the exponential integrators \cite{su2020deuflhard} and splitting methods \cite{zhang2017second}. However, as a result of the fourth-order spatial derivative in \eqref{GB}, these traditional schemes can not reach their ideal convergence rates when the solution is not smooth enough. For example, an explicit finite difference scheme was constructed in \cite{ortega1990nonlinear}, which strictly requires the boundedness of $\pa_x^6 z$ and $\pa_t^4 z$ and a time step restriction of $\Delta t=O(\Delta x^2)$, where $\Delta t$ and $\Delta x$ represent the time and space step, respectively. Unfortunately, the solutions involved in practical applications become rough due to the interference of noise. Thus it is necessary to find appropriate methods which can achieve the ideal convergence even for rough solutions. To this aim, some low regularity exponential integrators (LREIs) requiring low additional regularity have been established by introducing the concepts of \emph{twisted variable} $w(t):={\rm e}^{it\pa_x^2}u(t)$ and Duhamel's formula, see \cite{li2022lowregularity,ostermann2019two}. Compared to the classical methods, e.g., classical exponential integrators, these strategies give rise to some numerical schemes that still converge even when the solution is rough. Specifically, Ostermann \& Su \cite{ostermann2019two} gave a first-order and a second-order LREIs and obtain the linear and quadratic convergence in $H^r$ ($r>1/2$) by requiring one and three additional derivatives, respectively. This demand is weaker than that of the operator splitting method \cite{zhang2017second} and the spectral method \cite{cheng2015fourier}, the latter of which requires the boundedness of at least four additional temporal and spatial derivatives to attain the second-order convergence in time. Recently, the authors \cite{li2022lowregularity} proposed a new first-order and second-order LREIs, which converge with less additional derivatives required than those in \cite{ostermann2019two}. In particular, the second-order LREI in \cite{li2022lowregularity} converges quadratically with two additional derivatives required, which is weaker than that of \cite{ostermann2019two}.

In this article, we will introduce a newly developed LREI which has first-order accuracy in $H^r$ by requiring the boundedness of additional spatial derivatives at the order of $p(r)$, where $p(r)$ is non-increasing with respect to $r$, i.e.,
\[\|u^n-u(t_n)\|_r\lesssim \tau,\quad \mathrm{for}\quad u\in L^\infty(0, T; H^{r+p(r)}).\]
Particularly, $p(r)=0$ for all $r>5/2$, which means the method is convergent at the first order in $H^r$ with no additional regularity needed.
The first-order LREI is established by the following strategy:
\begin{enumerate}[ (i)]
	\item In the first step, we rewrite the GB equation as a first-order system
	$$\begin{pmatrix}
		z\\
		z_t
	\end{pmatrix}_t=\begin{pmatrix}
		0&\ \ 1\\
		-\pa_x^4+\pa_x^2& \ \ 0
	\end{pmatrix}\begin{pmatrix}
		z\\
		z_t
	\end{pmatrix}+\begin{pmatrix}
	0\\
	(z^2)_{xx}
\end{pmatrix}.$$
Then we diagonalize the above matrix in Fourier space and introduce a new complex variable $u(t)$ involving $z$ and $z_t$ so that the GB equation equivalents to a Schr\"odinger-type equation.

	\item In the second step, we extract the dominant term in the linear part of the equation involving $u(t)$ and introduce the so-called \emph{twisted variable} \[w(t)={\rm e}^{it\pa_x^2 }u(t).\]
An appropriate approximation is used to integrate the Duhamel's formula on the new variable $w$.

	\item Finally, we twist the variable back and obtain an approximation to $u$.  The integral for the nonlinear term is approximated so that the iteration can be efficiently calculated in physical space or Fourier space.
\end{enumerate}
\begin{remark}
	The method of twisted variable is firstly introduced by Ostermann and Schratz \cite{Ostermann2018} to design low-regularity numerical schemes for the nonlinear Schr\"odinger equation. Since then this technique has been extensively applied for the nonlinear Schr\"odinger equation \cite{Ostermann2019,li2021fully,Ostermann2021,wu2020first,ostermann2022fully}, KdV equation \cite{hofmanova2017exponential,ostermann2020lawson,wu2019optimal,wu2022embedded}, Klein-Gordon equation \cite{baumstark2018uniformly,wang2022symmetric} and other equations \cite{rousset2021general,schratz2021low}. Compared to classical numerical methods, this type of low regularity integrators can achieve the same  convergence when the solutions are less regular.
\end{remark}

Below we present our idea to design the new LERI briefly. The approach is based on the phase space analysis of the nonlinear dynamics. Specifically, we are devoted to finding a suitable approximation for the following time integral
\[\int_0^{\tau}{\rm e}^{-is(k^2+k_1^2+k_2^2)}ds, \quad {\rm with} \quad k_1+k_2=k.\]
The leading term $-2k^2$ is kept and integrated exactly in \cite{ostermann2019two}, i.e.,
\[\int_0^{\tau}{\rm e}^{-is(k^2+k_1^2+k_2^2)}ds=\int_0^{\tau}{\rm e}^{-2isk^2+2isk_1k_2}ds\approx \int_0^{\tau}{\rm e}^{-2isk^2}ds.\]
This finally leads to a first-order scheme with one additional order of regularity required \cite{ostermann2019two}. To weaken the constraint on regularity, the authors applied the identity $1=\frac{k_1+k_2}{k}$ together with the property
\begin{align}\label{idea2}
	k^2+k_1^2+k_2^2=2k_2^2+2kk_1=2k_1^2+2kk_2=2k^2-2k_1k_2,
\end{align}
and decompose the integral as
\begin{align*}
\int_0^{\tau}&{\rm e}^{-is(k^2+k_1^2+k_2^2)}ds=\frac{k_1}{k}\int_0^{\tau}{\rm e}^{-is(k^2+k_1^2+k_2^2)}ds+\frac{k_2}{k}\int_0^{\tau}{\rm e}^{-is(k^2+k_1^2+k_2^2)}ds\\
&= \frac{k_1}{k}\int_0^{\tau}{\rm e}^{-2is(k_2^2+kk_1)}ds+\frac{k_2}{k}\int_0^{\tau}{\rm e}^{-2is(k_1^2+kk_2)}ds\\
&\approx \frac{k_1}{k}\int_0^{\tau}\left({\rm e}^{-2isk_2^2}+{\rm e}^{-2iskk_1}-1\right)ds+\frac{k_2}{k}\int_0^{\tau}\left({\rm e}^{-2isk_1^2}+{\rm e}^{-2iskk_2}-1\right)ds,\end{align*}
where the the integrals in the last line can be integrated exactly in phase space.
In this work, we apply the identity
\begin{align}\label{idea1}
	\frac{k_1^2+k_2^2+2k_1k_2}{k^2}=1
\end{align}
instead and the decomposition follows as
\begin{align}
{\rm e}^{-is(k^2+k_1^2+k_2^2)}&=\frac{k_1^2}{k^2}{\rm e}^{-2is(k_2^2+kk_1)}+\frac{k_2^2}{k^2}{\rm e}^{-2is(k_1^2+kk_2)}+\frac{2k_1k_2}{k^2}{\rm e}^{-2is(k^2-k_1k_2)}\notag\\
&\approx\frac{k_1^2}{k^2}({\rm e}^{-2isk_2^2}+{\rm e}^{-2iskk_1}-1)+\frac{k_2^2}{k^2}({\rm e}^{-2isk_1^2}+{\rm e}^{-2iskk_2}-1)\notag\\
	&\quad+\frac{2k_1k_2}{k^2}({\rm e}^{-2isk^2}+{\rm e}^{2isk_1k_2}-1),\label{idea2}
\end{align}
where all three terms in the approximation can be exactly integrated. In this way we are able to establish the numerical flow as follows
\be\label{znscheme}
z^n=\frac{1}{2}(u^n+\overline{u^n})+at_n+b,\quad
	z^n_t=\frac{i}{2}\langle\partial_x^2\rangle(u^n-\overline{u^n})+a,
\ee
where
\begin{align}\label{initial}
	a=\mathcal{F}_0(z_{t}(0, \cdot))=\frac{1}{2\pi}\int_{\mathbb T} \psi_0(x)dx, \quad 	b=\mathcal{F}_0(z_{}(0, \cdot))=\frac{1}{2\pi}\int_{\mathbb T} \phi_0(x)dx,
\end{align}
and
\be\label{firstsch}
u^{n+1}=\Psi_1^{\tau}(u^{n}),\quad n\ge 0, \quad u^0=u(0, x)=\phi(x)-b-i\la \pa_x^2 \ra^{-1}(\psi(x)-a),
\ee
with $\la\pa_x^2 \ra^{-1}$ defined in Section 2 and
\begin{align}\label{sch0}
	&\Psi_1^{\tau}(f)={\rm e}^{i\tau \langle \partial_x^2 \rangle}f-\frac{i}{4}B^{\tau}\bigg[-i\pa_x^{-2}\left[\left({\rm e}^{2i\tau \pa_x^2}\pa_x^{-2}\ov{f}\right)\left(\pa_x^2\ov{f}\right)\right]+i\pa_x^{-2}\left[\left(\pa_x^{-2}\ov{f}\right)^2\right]\notag\\
	&-i{\rm e}^{i\tau \pa_x^2}\pa_x^{-3}\left[\left({\rm e}^{i\tau \pa_x^2}\pa_x^{}\ov{f}\right)\left({\rm e}^{-i\tau \pa_x^2}\pa_x^{}\ov{f}\right)\right]+i\pa_x^{-3}\left[\left(\pa_x^{}\ov{f}\right)\ov{f}\right]-2\tau\pa_x^{-2}\left[\left(\pa_x^{2}\ov{f}\right)\ov{f}\right]\notag\\
	&-i\pa_x^{-4}\left({\rm e}^{2i\tau \pa_x^2}-1\right)\left(\pa_x\ov{f}\right)^2+i\pa_x^{-2}{\rm e}^{-i\tau \pa_x^2}\left({\rm e}^{i\tau \pa_x^2}\ov{f}\right)^2-i\pa_x^{-2}\left(\ov{f}\right)^2\notag\\
	&-2\tau\pa_x^{-2}\left(\pa_x\ov{f}\right)^2+\frac{i}{2}\Big[(\partial_x^{-1}f)^2-{\rm e}^{i\tau \partial_x^2}({\rm e}^{-i\tau \partial_x^2}\partial_x^{-1}f)^2\Big]\notag\\
	&-i{\rm e}^{i\tau \partial_x^2}\partial_x^{-1}\Big[({\rm e}^{-i\tau \partial_x^2}f)({\rm e}^{i\tau \partial_x^2}\partial_x^{-1}\overline{f})\Big]+i\partial_x^{-1}\big[f(\partial_x^{-1}\overline{f})\big]\bigg]\notag\\
	&-i\tau (at_n+b)B^{\tau}\big(f +\psi_1(2i\tau \partial_x^2)\overline{f}\big),
\end{align}
with $\psi_1$ and $B^\tau$ given by \eqref{psi1} and \eqref{BL0} with \eqref{AB}, respectively. It can be easily seen that the scheme is explicit and easy to implement if one applies Fourier spectral method for spatial discretization.

Now we state the main theorem concerning the convergence of the above scheme \eqref{sch0}. Before that, we define a function $p(r)$:
\[p(r)=
\left\{
\begin{aligned}
&1,\quad &r=1;\\
&(3-2r)+,\quad &1<r\le7/6;\\
&2/3,\quad &7/6<r\le17/12;\\
&(7/2-2r)+,\quad &17/12<r\le3/2;\\
&5/4-r/2,\quad &3/2<r< 5/2;\\
&0+,\quad &r=5/2;\\
&0,\quad &r\ge5/2,
\end{aligned}\right.\]
where $c+$ means $c+\varepsilon$ for any sufficiently small $\varepsilon>0$.

\begin{theorem}\label{1orderth1}
For $r\ge 1$, suppose that the exact solution of \eqref{GB} satisfies $z$ $\in$ $C(0,T; H^{r+p(r)})$ and $z_t$ $\in$ $C(0,T; H^{r+p(r)-2})$. Then there exists a constant $\tau_0>0$ such that for all step size $0<\tau\leq\tau_0$ and $t_n\leq T$, we have
	\begin{align*}
		\lVert z(t_n)-z^n\rVert_r+\lVert z_t(t_n)-z_t^n\rVert_{r-2}\leq C \tau^{},
	\end{align*}
	where $C>0$ depends on $T$, $\lVert z\rVert_{L^{\infty}(0,T; H^{r+p(r)})}$ and $\lVert z_t \rVert_{L^{\infty}(0,T;H^{r+p(r)-2})}$.
\end{theorem}

\begin{figure}[htbp]
	\center
	\hspace{-110pt}
	\subfigure{
		\begin{minipage}[c]{0.5\linewidth}
			\centering
			\includegraphics[height=6cm, width=1.68\linewidth]{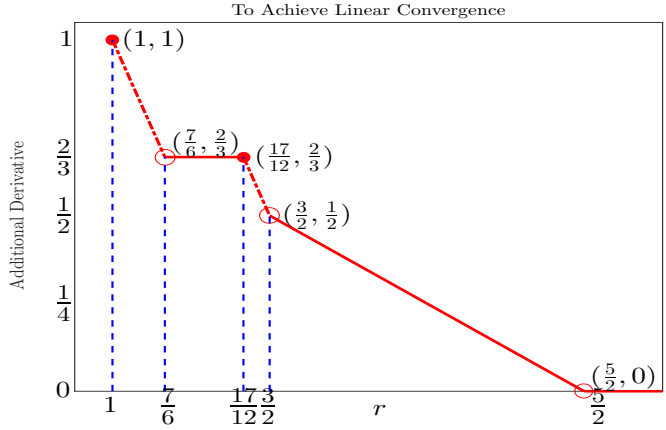}
		\end{minipage}
	}
	\caption{Additional order of regularity required to achieve the first-order convergence. Particularly, for the domain $(1, 7/6]$ and $(17/12, 3/2]$, we plot it by  a dash-dotted line or hollow points to mean that a plus sign exists in $p(r)$, e.g., the scheme is convergent linearly in $H^{3/2}$ for solutions in $H^{(3/2+1/2)+}$. } \label{addr}
\end{figure}

It is clear $p(r)$ represents the order of additional regularity required to
promise the first-order convergence of the numerical solution in $H^r$.
Fig. \ref{addr} displays the plot of $p(r)$, from which we observe that $p(r)$ is non-increasing. Particularly, $p(r)\equiv 0$ when $r>5/2$, which means the scheme is convergent linearly in $H^r$ without loss of regularity when $r>5/2$.

Compared to the convergence results of the scheme in \cite{ostermann2019two}, which converges in $H^r$ ($r>1/2$) at the first order when the solution belongs to $H^{r+1}$, and the method in \cite{li2022lowregularity}, which achieves the first-order convergence in $H^r$ for $r>7/6$ as the solution lies in $H^{r+2/3}$, it is obvious that our newly proposed scheme \eqref{sch0} requires less regularity to attain the ideal first-order convergence. Furthermore, for the convergence without smoothness assumptions, i.e., $p(r)=0$, compared to the first-order LREIs proposed in \cite{li2022lowregularity} and \cite{ostermann2019two} which converge at the order of $1/2$ and $\frac{r-1/2}{3r+1/2}-$, respectively, our newly developed first-order LREI presents a linear convergence without additional regularity assumptions when $r>5/2$. This is the first time to establish the optimal linear convergence without loss of regularity for the GB equation. On the other hand, we have to admit that the deficiency is that we have to impose $r\ge 1$ due to the stability analysis (cf. Section 4). Thus the analysis in $H^r$ for $r\le 1$ is absent at the moment.

 The rest of the paper is organized as follows. In Section \ref{sec1}, we present some notions and powerful technical tools. The first-order LREI is constructed in Section \ref{3}. Section \ref{4} is devoted to establishing the error estimate of the scheme. Some numerical results are presented to confirm the theoretical analysis in Section \ref{num} and conclusions are drawn in Section \ref{conclu}.

\section{Preliminary}\label{sec1}
In this section, we introduce some notations and present some useful technical lemmas which are of vital importance to design the method or to establish the error estimates.

\subsection{Notations}
In this paper, we use the notation $X\lesssim Y$ to denote that there exists a constant $C>0$ which may be different from line to line but is independent of the time step $\tau$ such that $\lv X\rv\le C Y$. The Fourier transform of a function $f$ on a torus $\mathbb{T}$ is defined by the coordinate representation $\{\widehat{f}_k \}_{k=-\infty}^{+\infty}$ under the basis $\{{\rm e}^{ikx}\}_{k=-\infty}^{+\infty}$ in $L^2(\mathbb{T})$, where
\[\mathcal{F}_k(f)=\widehat{f}_k=\frac{1}{2\pi}\int_{\mathbb{T}}f(x){\rm e}^{-ikx}dx.\]
Thus $f(x)=\sum\limits_{k\in \mathbb{Z}}\widehat{f}_k {\rm e}^{ikx}$ is the inverse Fourier transform. The norm and inner product in $L^2$ are defined respectively by
\be\label{L2}
\|f\|:=\|f\|_{L^2}=\big(\sum\limits_{k\in \mathbb{Z}}\lvert \widehat{f}_k\rvert^2\big)^{1/2},\quad ( f, g) =\sum\limits_{k\in\mathbb{Z}}\widehat{f}_k\overline{\widehat{g}_k}=\frac{1}
{2\pi}\int_\mathbb{T} f(x)\overline{g(x)}dx.\ee
Moreover, we define several operators given in Fourier space as
\be\label{def1}
\begin{split}
&\partial_x^{-1}f=\sum\limits_{k\neq 0}\frac{1}{ik}
\widehat{f}_k {\rm e}^{ikx},\quad \widetilde{f}=\sum\limits_{k\in\mathbb{Z}}\big\rvert\widehat{f}_{k}
\big\lvert{\rm e}^{ikx},\\
&\lvert\partial_x\rvert^\alpha f=\sum\limits_{k\neq 0}\lvert k\rvert^\alpha\widehat{f}_k {\rm e}^{ikx},\quad
J^\alpha f=\sum\limits_{k\in\mathbb{Z}}(1+k^2)^{\alpha/2}\widehat{f}_k {\rm e}^{ikx},\quad \alpha\in\mathbb{R}.
\end{split}
\ee
Similarly, we define $\langle \partial_x^2 \rangle:=\sqrt{-\partial_{x}^{2}+\partial_{x}^{4}}$ and its inverse by
\[\langle \partial_x^2 \rangle f=\sum\limits_{k\in \mathbb{Z}}\sqrt{k^2+k^4}
\widehat{f}_k {\rm e}^{ikx},\quad
\langle \partial_x^2 \rangle ^{-1}f=\sum\limits_{k\neq 0}\frac{1}{\sqrt{k^2+k^4}}\widehat{f}_k {\rm e}^{ikx}.\]

Furthermore, we introduce the Sobolev space
$H^\alpha$ with $\alpha\in \mathbb{R}$, which consists of the functions
$f=\sum\limits_{k\in \mathbb{Z}}\widehat{f}_k {\rm e}^{ikx}$ such that
$\|f\|_\alpha=\|J^\alpha f\|<\infty$, where
\begin{align*}
	\| f\|_{\alpha}^2=\|J^\alpha f\|^2=\sum\limits_{k\in \mathbb{Z}} (1+k^2)^\alpha \lvert \widehat{f}_k\rvert^2.
\end{align*}
It is clear that for $f$ with zero mean value, i.e., $\widehat{f}_k=0$, it holds $\| f\|_{\alpha}\lesssim \|\lvert\partial_x\rvert^\alpha f\|\lesssim \| f\|_{\alpha}$.
For $\alpha=0$, it is clear that the space reduces to $L^2$ and the corresponding norm is simply denoted as $\|\cdot\|$ which agrees with \eqref{L2}.

We say that $R=R(u,t,\tau,\xi)\in \mathcal{R}_{\theta}(\tau^{\nu})$ if and only if
\[
\|R(u,t,\tau,\xi)\|_r\le  C\tau^{\nu},\]
where $R(u,t,\tau,\xi)$ depends on the value $u(t+\xi)$, $0\le \xi\le\tau$, and $C$ relies on $\mathop{{\rm sup}}\limits_{0\leq s \leq \tau}	\left\lVert u(t+s)\right\rVert_{r+\theta}$. We write $f=g+\mathcal{R}_{\theta}(\tau^{\nu})$ whenever $f=g+R$ with $R\in \mathcal{R}_{\theta}(\tau^{\nu})$.

\subsection{Preliminary tools}
To begin with, we introduce the Kato-Ponce inequalities, which was previously proved in \cite{Kato1988CommutatorEA,bourgain2014endpoint,li2019kato} in the whole space $\mathbb{R}$ and extended to the periodic case by Li and Wu \cite{LiWu22} recently.
\begin{lemma} (The Kato-Ponce inequalities)\label{lemma1}
	(i) If $r$ $>$ $1/2$ and $f,g$ $\in$ $H^{r}$, then we have
	\be\label{kp}
	\|fg\|_r\lesssim 	\|f\|_{r} 	\|g\|_{r}.\ee
	\noindent(ii)  If $s$ $>$ $0$, $1<p\le \infty$, $1<p_1, p_3< \infty$, $1<p_2, p_4\le \infty$ satisfying $\frac{1}{p}=\frac{1}{p_1}+\frac{1}{p_2}$ and $\frac{1}{p}=\frac{1}{p_3}+\frac{1}{p_4}$, then we have the following inequality
	\be\label{kp1}
	\left\|J^s (fg)\right\|_{L^p}\lesssim \left\|J^sf\right\|_{L^{p_1}}\left\|g\right\|_{L^{p_2}}+\left\|J^sg\right\|_{L^{p_3}}\left\|f\right\|_{L^{p_4}}.\ee
\end{lemma}

 Next we present Hardy-Littlewood-Sobolev type inequality and Sobolev embedding theorem on the torus $\mathbb{T}$, which provides a new approach for the subsequent estimate of local truncation errors. We refer to \cite{adams2003sobolev,ambrosio2015periodic,benyi,maz2002bourgain,stein2016introduction} and references therein.
\begin{lemma}\label{sobolevcompact}
		(i) (Hardy-Littlewood-Sobolev type inequality) Let $s\in[0,1/2)$. Then there exists a constant $C=C(s)>0$ such that
	\[\|f\|_{-s}\leq C\|f\|_{L^{\frac{2}{1+2s}}(\mathbb{T})},\]
	for any $f\in L^{\frac{2}{1+2s}}(\mathbb{T})$.
	
	\noindent (ii) (Sobolev embedding theorem) Let $s\in(0,1/2)$. The inclusion
	 \[H^s(\mathbb{T})\subseteq L^q(\mathbb{T})\]
	  is continuous for any $q\in\left[1, \frac{2}{1-2s}\right]$.
\end{lemma}
\begin{lemma}\label{lemma2}
	(i) For all $x,y$ $\in$ $\mathbb{R}$ and $0\leq \theta\leq1$, we have
	\begin{align*}
		&\lvert {\rm e}^{ix}-1\rvert\leq2^{1-\theta}\rvert x\rvert^{\theta}, \quad \lvert {\rm e}^{ix}-1-ix\rvert \leq 2^{1-2\theta}\rvert x\rvert^{1+\theta}.
	\end{align*}
	\noindent(ii) For $t\in\mathbb{R}$, $r\ge 0$ and $f\in H^r$, we have
	\[\|\psi_1(it\pa_x^2)f\|_r\le \|f\|_r,\]
where
\begin{align}\label{psi1}
	\psi_1(y)=\int_0^1 {\rm e}^{ys}ds, \quad \mathrm{for}
	\quad y\in\mathbb{C}.
\end{align}

\end{lemma}
For the details of the proof, we refer to \cite{ostermann2019two}. Moreover, we illustrate a lemma which was introduced by \cite{li2022lowregularity,li2021fully}.

\begin{lemma}\label{lemma3}
	\noindent (i) For $f \in H^{r}$ with $r\ge 0$, $t\in \mathbb{R}$, it holds
	\be\label{l241}
	\lVert \langle \partial_x^2 \rangle^{-1} f\rVert_r\leq \lVert f \rVert_{r},    \ \lVert Af\rVert_r\leq \lVert f \rVert_{r},  \ \lVert Bf\rVert_r\leq \lVert f \rVert_{r}, \ \lVert ({\rm e}^{itA}-1)f\rVert_r\leq \lvert t\rvert  \lVert f \rVert_{r},\ee
	where $A$ and $B$ are given by
\begin{align}\label{AB}
	A:=\langle \partial_x^2 \rangle+\partial_{x}^2,\quad B:=\langle \partial_x^2 \rangle^{-1}\partial_{x}^2.
\end{align}
	
	\noindent (ii) For $r\ge 0$ and $0\leq\gamma\leq1$, $f\in H^{r+2\gamma}$, one has
	\be\label{l242}
	\lVert ({\rm e}^{it \partial_x^2}-1)f\rVert_r \lesssim \lvert t \rvert^{\gamma} \lVert  f \rVert_{r+2\gamma},\quad
	\lVert ({\rm e}^{it \langle\partial_x^2\rangle}-1)f\rVert_r \lesssim \lvert t \rvert^{\gamma} \lVert  f \rVert_{r+2\gamma}.
	\ee

	\noindent (iii) If $f,g$ $\in$ $H^{1}$, then it holds
	\be\label{l244}
	\big\| J^{-1} \left(g (J f)  \right)\big\|\lesssim {\rm min}\{\lVert f\rVert \lVert g \rVert_1, \lVert f\rVert_1 \lVert g \rVert\}.
	\ee
	\noindent (iv) If $f,g$ $\in$ $H^{r}$, $r>1/2$ then we have
	\be\label{l2444}
	\big\| J^{-1} \left(g (J f)  \right)\big\|_r\lesssim \|f\|_r \|g\|_r.
	\ee
\end{lemma}

\begin{lemma}\label{newb}
	For $f,g\in H^{r}$ with $r\ge 1$, it holds
	\be\label{l243}
	\big\lVert \left\lv\partial_x\right\rv^{-2} \big[\left( \left\lv\partial_x\right\rv g)( \left\lv\partial_x\right\rv f  \right)\big]\big\rVert_r \lesssim \lVert f\rVert_r \lVert g \rVert_r.
	\ee
\end{lemma}
\begin{proof}
	To show the above inequality for $r>1$, we only need to confirm
	\[\big\lVert \left\lv\partial_x\right\rv^{r-2} \big[\left( \left\lv\partial_x\right\rv g)( \left\lv\partial_x\right\rv f  \right)\big]\big\rVert \lesssim \lVert f\rVert_r \lVert g \rVert_r.\]
	According to the duality principle in $L^2$, it suffices to prove
	\[\left(\left\lv\partial_x\right\rv^{r-2} \big[\left( \left\lv\partial_x\right\rv g)( \left\lv\partial_x\right\rv f  \right)\big], \phi\right)\lesssim \| f\|_r \| g \|_r \|\phi\|, \quad \forall \phi \in L^2,\]
	which is equivalent to
	\[\sum\limits_{k\neq 0}\sum\limits_{k_1+k_2=k}\lv k\rv^{r-2}\lv k_1\rv\lv k_2\rv\widehat{f}_{k_1}\widehat{g}_{k_2}\overline{\widehat{\phi}_{k}}\lesssim \|f\|_r\|g\|_r\|\phi\|.\]
	To this aim, we divide the above formula into two parts by discussing the relationship between Fourier coefficients $k$ and $k_1$, i.e.,
	\begin{align}\label{kk1}
		\sum\limits_{k\not=0}\sum\limits_{k_1+k_2=k}\lv k\rv^{r-2}\lv k_1\rv\lv k_2\rv\widehat{f}_{k_1}\widehat{g}_{k_2}\overline{\widehat{\phi}_{k}}
		&= \sum\limits_{k\not=0}\sum\limits_{k_1+k_2=k \atop \lv k_1\rv \le 2\lv k\rv}\lv k\rv^{r-2}\lv k_1\rv\lv k_2\rv\widehat{f}_{k_1}\widehat{g}_{k_2}\overline{\widehat{\phi}_{k}}\notag\\
&\hspace{-6mm}+\sum\limits_{k\not=0}\sum\limits_{k_1+k_2=k \atop \lv k_1\rv > 2\lv k\rv}\lv k\rv^{r-2}\lv k_1\rv\lv k_2\rv\widehat{f}_{k_1}\widehat{g}_{k_2}\overline{\widehat{\phi}_{k}}.
	\end{align}
	For the first term in \eqref{kk1}, we have
	\[\lv k_2\rv=\lv k-k_1\rv\le \lv k\rv+\lv k_1\rv\le 3\lv k\rv.\]
	By using Plancherel's identity and the bilinear estimate, the first term can be bounded as
	\begin{align*}
		\sum\limits_{k\not=0}\sum\limits_{k_1+k_2=k \atop \lv k_1\rv \le 2\lv k\rv}\lv k\rv^{r-2}\lv k_1\rv \lv k_2\rv\widehat{f}_{k_1}\widehat{g}_{k_2}
		\overline{\widehat{\phi}_{k}}
		&\lesssim \sum\limits_{k\not=0}\sum\limits_{k_1+k_2=k \atop \lv k_1\rv \le 2\lv k\rv}\lv k\rv^{r}\big\lv\widehat{f}_{k_1}\big\rv\left\lv\widehat{g}_{k_2}\right\rv
		\big\lv\overline{\widehat{\phi}_{k}}\big\rv\\
		&\hspace{-3cm}\lesssim (\lv\pa_x\rv^{r}(\widetilde{f}\widetilde{g}),\widetilde{\phi})
	\lesssim \|\widetilde{f}\widetilde{g}\|_r\|\widetilde{\phi}\|\lesssim \|\widetilde{f}\|_r \|\widetilde{g}\|_r\|\widetilde{\phi}\|\lesssim \|f\|_r\|g\|_r\|\phi\|,
	\end{align*}
	where $\widetilde{f}, \widetilde{g}$ and $\widetilde{\phi}$ are defined in  \eqref{def1}.
	
	For the second term in \eqref{kk1}, thanks to $\lv k_1\rv > 2\lv k\rv$, we are led to
	\[\lv k_2\rv=\lv k_1-k\rv\ge \lv k_1\rv-\lv k\rv>\lv k\rv.\]
	For $r>1$, it holds
	\[\lv k\rv^{r-2}\lv k_1\rv\lv k_2\rv=\lv k\rv^{-r}\lv k\rv^{2r-2}\lv k_1\rv\lv k_2\rv\lesssim \lv k\rv^{-r}\lv k_1\rv^{r}\lv k_2\rv^{r},\]
	which implies
	\begin{align*}
		&\sum\limits_{k\not=0}\sum\limits_{k_1+k_2=k \atop \lv k_1\rv > 2\lv k\rv}\lv k\rv^{r-2}\lv k_1\rv\lv k_2\rv\widehat{f}_{k_1}\widehat{g}_{k_2}\overline{\widehat{\phi}_{k}}
		\lesssim \sum\limits_{k\not=0}\sum\limits_{k_1+k_2=k \atop \lv k_1\rv > 2\lv k\rv}\lv k\rv^{-r}\lv k_1\rv^{r}\lv k_2\rv^{r}\big\lv\widehat{f}_{k_1}\big\rv\left\lv\widehat{g}_{k_2}\right\rv\big\lv\overline{\widehat{\phi}_{k}}\big\rv\\
		&\quad\lesssim \sum\limits_{k\not=0}\mathcal{F}_k(\lv\pa_x\rv^{r}\widetilde{f}\lv\pa_x\rv^{r}\widetilde{g})\lv k\rv^{-r}\lv\widehat{\phi}_k\rv\\
		&\quad\lesssim \mathop{{\rm max}}\limits_{k}\left\lv\int_{\mathbb{T}}\lv\pa_x\rv^{r}\widetilde{f}(x)\lv\pa_x\rv^{r}\widetilde{g}(x)e^{-ikx}dx\right\rv\sum\limits_{k\not=0}\lv k\rv^{-r}\lv\widehat{\phi}_k\rv\\
		&\quad\lesssim \|\lv\pa_x\rv^{r}\widetilde{f}\lv\pa_x\rv^{r}\widetilde{g}\|_{L^1}\|(\lv k\rv^{-r})_{0\not=k\in \mathbb{Z}}\|_{l^2}\|(\lv\widehat{\phi}_k\rv)_{0\not=k\in \mathbb{Z}}\|_{l^2}\\
		&\quad\lesssim \|\lv\pa_x\rv^{r}\widetilde{f}\| \|\lv\pa_x\rv^{r}\widetilde{g}\| \|\widetilde{\phi}\|\\
		&\quad\lesssim \|f\|_r \|g\|_r\|\phi\|.
	\end{align*}
	The proof is completed for the case of $r>1$.
	For the case of $r=1$, by using the result in Lemma \ref{lemma3} (iii)
	\[\| \lv\pa_x\rv^{-1}(g\lv\pa_x\rv f )\| \lesssim \| J^{-1}(g\lv\pa_x\rv f )\|\lesssim \| J^{-1}(\widetilde{g}(J \widetilde{f}) )\|\lesssim \|f\|_1\|g\|,\]
	we are led to
	\begin{align*}
		\| \left\lv\partial_x\right\rv^{-2} \big[\left( \left\lv\partial_x\right\rv g)( \left\lv\partial_x\right\rv f  \right)\big]\|_1&\lesssim \| \left\lv\partial_x\right\rv^{-1} \big[\left( \left\lv\partial_x\right\rv g)( \left\lv\partial_x\right\rv f  \right)\big]\|\notag\\
&\lesssim \|f\|_1\|\lv\pa_x\rv g\|\lesssim \|f\|_1\|g\|_1.
	\end{align*}
	This completes the proof.
\end{proof}

\section{First-order exponential-type integrator $\Psi_1^{\tau}$}\label{3}
In the following part, we construct the first-order LREI based on the idea in \eqref{idea1}--\eqref{idea2}.
\subsection{Homogenization and reformulation of the GB equation}
As can be seen blow, we will frequently encounter the operator $\pa_x^{-1}$ or $\pa_x^{-2}$ during the process of integration, which makes the mean value of the obtained function zero. Hence usually the zero-mode needs to be treated separately. Fortunately, thanks to the periodic boundary conditions, the zero-mode of $z$ can be integrated exactly so that it remains to investigate other nonzero Fourier modes.

By the periodicity of the solution, one easily gets
\begin{align*}
	\mathcal{F}_0(z_{tt})=\partial_{tt}\mathcal{F}_0(z)=0,
\end{align*}
which immediately gives $\mathcal{F}_0(z)=at+b$, where $a$ and $b$ are defined as \eqref{initial}.
Setting $z=\mathcal{F}_0(z)+\check{z}$ and plugging it into \eqref{GB}, we derive that
\be\label{HGB}
\left\{\begin{aligned}
&\check{z}_{tt}+\check{z}_{xxxx}-(2at+2b+1)\check{z}_{xx}-(\check{z}^2)_{xx}=0, \quad x \in \mathbb{T}, \quad t>0,\\
&\check{z}(0, x)=\phi(x)-b,\quad \check{z}_{t}(0, x)=\psi(x)-a.
\end{aligned}\right.
\ee
Diagonalize the equivalent first-order system
	$$\begin{pmatrix}
		\check z\\
		\check{z}_t
	\end{pmatrix}_t=\begin{pmatrix}
		0&\ \ 1\\
		-\pa_x^4+\pa_x^2& \ \ 0
	\end{pmatrix}\begin{pmatrix}
		\check z\\
		z_t
	\end{pmatrix}+\begin{pmatrix}
	0\\
	(\check{z}^2)_{xx}+(2at+2b)\check{z}_{xx}
\end{pmatrix},$$
and set
\[\la \pa_x^2 \ra=\sqrt{\pa_x^2+\pa_x^4}, \quad u=\check{z}-i\la \pa_x^2 \ra^{-1}\check{z}_t, \quad v=\ov{\check{z}}-i\la \pa_x^2 \ra^{-1}\ov{\check{z}_t},\]
we are led to the following coupled system
\be\label{GBcoup}
\left\{	\begin{aligned}
	&i\partial_{t}u =-\langle \partial_x^2 \rangle  u +  B\left[ \frac{1}{4}(u + \bar{v})^2+(at+b)(u + \bar{v})\right],    \\
	&i\partial_{t}v =-\langle \partial_x^2 \rangle  v + B \left[ \frac{1}{4}(\bar{u} + v)^2+(at+b)(\bar{u} + v)\right],
\end{aligned}
\right.\ee
where the operator $B$ is defined in \eqref{AB}. Recalling that $z$ is a real function, this implies $u=v$ and \eqref{GBcoup} reduces to a single first-order equation involving a complex variable
\be\label{GBfirst}\left\{
\begin{aligned}
	&i\partial_{t}u =-\langle \partial_x^2 \rangle  u +  B \left[ \frac{1}{4}(u + \overline{u})^2+(at+b)(u + \overline{u})\right],\\
&u(0,x)=\check{z}(0,x)-i\la \pa_x^2 \ra^{-1}\check{z}_t(0,x).
\end{aligned}\right.
\ee
While $\check{z}$ and $\check{z}_t$ can be recovered through
\be\label{uvz}
\check{z}=\frac{1}{2}(u+\ov{u}), \quad
	\check{z}_t=\frac{i}{2} \la \pa_x^2 \ra^{} (u-\ov{u}).
\ee

Noticing that the leading term of $\la \pa_x^2 \ra$ is $-\pa_x^2$, we introduce the so-called $twisted$ $variable$
\[w(t)={\rm e}^{it\pa_x^2}u(t).\]
Plugging it into $\eqref{GBfirst}$ yields
\be\label{aftertw}
\partial_t w = iAw-\frac{i}{4}{\rm e}^{it \partial_x^2 } B({\rm e}^{-it \partial_x^2 }w+{\rm e}^{it \partial_x^2 }\overline{w})^2-i(at+b){\rm e}^{it \partial_x^2 } B({\rm e}^{-it \partial_x^2 }w+{\rm e}^{it \partial_x^2 }\overline{w}).
\ee
Applying  Duhamel's formula of $\eqref{aftertw}$, we obtain
\begin{align}\label{duhamel}
	w(t_n+\sigma)&={\rm e}^{i\sigma A}w(t_n)-\frac{iB}{4}\int_0^{\sigma}{\rm e}^{i(\sigma-s) A}{\rm e}^{i(t_n+s) \partial_x^2}(g_1(w(t_n), s))^2ds\notag \\
	&\quad-iB\int_0^{\sigma}{\rm e}^{i(\sigma-s) A}{\rm e}^{i(t_n+s) \partial_x^2}[a(t_n+s)+b]g_1(w(t_n), s)ds,
\end{align}
where $A$ and $B$ are defined in \eqref{AB}, and
\[g_1(w(t_n), s)={\rm e}^{-i(t_n+s) \partial_x^2}w(t_n+s)+{\rm e}^{i(t_n+s) \partial_x^2}\overline{w(t_n+s)}.\]

Based on this, a first-order approximation can be easily derived \cite{ostermann2019two}
	\be\label{l5}
	\left\lVert w(t_n+\sigma)-w(t_n)\right\rVert_r \le C\sigma ,\quad r>1/2,
	\ee
	where $C$ only depends on $\mathop{{\rm sup}}\limits_{0\leq s \leq \sigma}	\left\lVert u(t_n+s)\right\rVert_r$.
Setting $\sigma=\tau$ and approximating $w(t_n+s)$ by $w(t_n)$ in the integral of \eqref{duhamel}, applying \eqref{l5} and Lemma \ref{lemma3} (i), we get
\begin{align}\label{duhamelfirst}
	w(t_n+\tau)&={\rm e}^{i\tau A}w(t_n)-\frac{i}{4}B{\rm e}^{i\tau A}\int_0^{\tau}{\rm e}^{i(t_n+s) \partial_x^2}(g_2(w(t_n), s))^2ds\notag \\
	&\quad-iB{\rm e}^{i\tau A}\int_0^{\tau}{\rm e}^{i(t_n+s) \partial_x^2}(at_n+b)g_2(w(t_n), s)ds+\mathcal{R}_{0}(\tau^2),
\end{align}
where $g_2(w(t_n), s)={\rm e}^{-i(t_n+s) \partial_x^2}w(t_n)+{\rm e}^{i(t_n+s) \partial_x^2}\overline{w(t_n)}$.

Twisting the variable back, we obtain an approximation of $u(t_n+\tau)$ with a local error of order two
\begin{align}
	u(t_n+\tau)&={\rm e}^{i\tau \langle \partial_x^2 \rangle}u(t_n)-\frac{i}{4}B{\rm e}^{i\tau \langle \partial_x^2 \rangle}\int_0^{\tau}{\rm e}^{is \partial_x^2}({\rm e}^{-is \partial_x^2}u(t_n)+{\rm e}^{is \partial_x^2}\overline{u}(t_n))^2ds\notag \\&\ -i(at_n+b)B{\rm e}^{i\tau \langle \partial_x^2 \rangle}\int_0^{\tau}{\rm e}^{is \partial_x^2}({\rm e}^{-is \partial_x^2}u(t_n) +{\rm e}^{is \partial_x^2}\overline{u}(t_n))ds+\mathcal{R}_{0}(\tau^2)\notag\\
	&={\rm e}^{i\tau \langle \partial_x^2 \rangle}u(t_n)-\frac{i}{4}B^{\tau}\big[I_0^{\tau}(u(t_n))+I_1^{\tau}(u(t_n))+2I_2^{\tau}(u(t_n))\big]\notag\\
	&\quad-i\tau (at_n+b)B^{\tau}\big(u(t_n) +\psi_1(2i\tau \partial_x^2)\overline{u}(t_n)\big)+\mathcal{R}_{0}(\tau^2),\label{uapp}
\end{align}
where $\psi_1$ is given by \eqref{psi1} and
 \begin{align}
	&B^{\tau}(f)=B{\rm e}^{i\tau \langle \partial_x^2\rangle}f=\langle \partial_x^2 \rangle^{-1}\partial_{x}^2{\rm e}^{i\tau \langle \partial_x^2\rangle}f, \quad	I_{0}^{\tau}(f)=\int_0^{\tau}{\rm e}^{is\partial_x^2}\big({\rm e}^{is\partial_x^2}\overline{f}\big)^2ds,\label{BL0}\\
	&I_{1}^{\tau}(f)=\int_0^{\tau}{\rm e}^{is\partial_x^2}\big({\rm e}^{-is\partial_x^2}f\big)^2ds,\qquad\,\, I_{2}^{\tau}(f)=\int_0^{\tau}{\rm e}^{is\partial_x^2}\big\lv{\rm e}^{-is\partial_x^2}f\big\rv^2ds.\label{L12}
\end{align}

Now we calculate the terms $I_j^\tau$ in \eqref{uapp} respectively.
Firstly for $f$ satisfying $\mathcal{F}_0(f)=0$, as was shown in \cite{ostermann2019two}, $I_{1}^{\tau}(f)$ and $I_{2}^{\tau}(f)$ can be calculated exactly as
\begin{align}
	I_{1}^{\tau}(f)&=\sum\limits_{k}\sum\limits_{k_1+k_2=k}\int_0^{\tau}{\rm e}^{-is(k^2-k_1^2-k_2^2)}ds \widehat{f}_{k_1}\widehat{f}_{k_2}{\rm e}^{ikx}\notag\\
	&=\Big(\sum\limits_{k}\sum\limits_{k_1+k_2=k \atop k_1\not=0,  k_2\not=0}\frac{{\rm e}^{-2i\tau k_1k_2}-1}{-2ik_1k_2}+\sum\limits_{k}\sum\limits_{k_1+k_2=k \atop k_1=0\ {\rm or} \ k_2=0}\int_0^{\tau}ds\Big)\widehat{f}_{k_1}\widehat{f}_{k_2}{\rm e}^{ikx}\notag\\
	&=\sum\limits_{k}\sum\limits_{k_1+k_2=k \atop k_1\not=0,  k_2\not=0}\frac{{\rm e}^{-2i\tau k_1k_2}-1}{-2ik_1k_2}\widehat{f}_{k_1}\widehat{f}_{k_2}{\rm e}^{ikx}+2\tau \widehat{f}_{0}\sum\limits_{k\in \mathbb{Z}}\widehat{f}_{k}{\rm e}^{ikx}-\tau\widehat{f}_{0}^2\notag\\
	&=\frac{i}{2}\bigg[(\partial_x^{-1}f)^2-{\rm e}^{i\tau \partial_x^2}({\rm e}^{-i\tau \partial_x^2}\partial_x^{-1}f)^2\bigg],\label{l1f}\\
	I_{2}^{\tau}(f)&=\sum\limits_{k_1, k_2\in\mathbb{Z}}
	\int_0^\tau e^{is(k_1^2-k_2^2-(k_1-k_2)^2)}ds \widehat{f}_{k_1}
	\overline{\widehat{f}_{k_2}}e^{i(k_1-k_2)x}\notag\\
	&\hspace{-5mm}=-\frac{i}{2}{\rm e}^{i\tau \partial_x^2}\partial_x^{-1}\big[({\rm e}^{-i\tau \partial_x^2}f)({\rm e}^{i\tau \partial_x^2}\partial_x^{-1}\overline{f})\big]+\frac{i}{2}\partial_x^{-1}\big[f(\partial_x^{-1}\overline{f})\big]+\tau \lVert f \rVert^2.\label{l2f}
\end{align}

It remains to calculate $I_0^\tau(f)$ which reads as
\be\label{I0}
I_0^\tau(f)=\sum\limits_{k\in \mathbb{Z}}\sum\limits_{k_1+k_2=k}\int_0^{\tau}{\rm e}^{-is\Phi}ds \widehat{\overline{f}}_{k_1}\widehat{\overline{f}}_{k_2}{\rm e}^{ikx} \quad {\rm with} \ \ \Phi=k^2+k_1^2+k_2^2.\ee
Different from the above two terms $I_1^\tau(f)$ and $I_2^\tau(f)$, in which the obtained integration is a function with separable variables $k$, $k_1$ and $k_2$ that enables us to compute the obtained convolution efficiently in physical space or Fourier space, it is impossible to compute the exact integral of $I_0^\tau$ efficiently in any space. To overcome this difficulty, in \cite{li2022lowregularity}, we proposed an approximation of $I_0^\tau$ by applying the identity  $1=\frac{k_1+k_2}{k}$ and an appropriate approximation which can be computed efficiently. In this paper, we utilize similar idea based on the identity
\[1=\frac{k_1^2+k_2^2+2k_1k_2}{k^2}, \quad \Phi=2k_2^2+2kk_1=2k_1^2+2kk_2=2k^2-2k_1k_2,\]
and approximate the corresponding integrals in a proper way.

\subsection{The first-order exponential-type integrator $\Psi_1^{\tau}$}

\noindent \textbf{Case I.} When $k=0$, the mean value of $I_0^\tau(f)$ can be computed exactly and efficiently by
\be\label{t0}
\mathcal{F}_0\left(I_0^\tau(f)\right)=\frac{i}{2}\mathcal{F}_0\big[\big({\rm e}^{i\tau \partial_x^2}\partial_x^{-1}\overline{f}\big)^2\big]-\frac{i}{2}\mathcal{F}_0
\big[\big(\partial_x^{-1}\overline{f}\big)^2\big]+\tau \big(\widehat{\ov{f}}_0\big)^2=T_0^{\tau}(f).
\ee
\noindent \textbf{Case II.} When $k\not=0$, in order to balance the power of  $k_1$ and $k_2$ as much as possible in the following estimations, we use different forms of the phase function $\Phi$ as $2k_2^2+2kk_1$, $2k_1^2+2kk_2$, and $2k^2-2k_1k_2$ for coefficients $\frac{k_1^2}{k^2}$, $\frac{k_2^2}{k^2}$, and $\frac{2k_1k_2}{k^2}$, respectively. Specifically,
\begin{align}\label{t1t2t3}
	\sum\limits_{k\not=0}\mathcal{F}_k\left(I_0^\tau(f)\right){\rm e}^{ikx}&=\sum\limits_{k\not=0}\sum\limits_{k_1+k_2=k}\int_0^{\tau}{\rm e}^{-is(k^2+k_1^2+k_2^2)}ds \widehat{\overline{f}}_{k_1}\widehat{\overline{f}}_{k_2}{\rm e}^{ikx}\notag\\
	&=\sum\limits_{k\not=0}\sum\limits_{k_1+k_2=k}\frac{k_1^2}{k^2}\int_0^{\tau}{\rm e}^{-2is(k_2^2+kk_1)}ds \widehat{\overline{f}}_{k_1}\widehat{\overline{f}}_{k_2}{\rm e}^{ikx}\notag\\
&\quad+\sum\limits_{k\not=0}\sum\limits_{k_1+k_2=k}\frac{k_2^2}{k^2}\int_0^{\tau}{\rm e}^{-2is(k_1^2+kk_2)}ds \widehat{\overline{f}}_{k_1}\widehat{\overline{f}}_{k_2}{\rm e}^{ikx}\notag\\
&\quad+\sum\limits_{k\not=0}\sum\limits_{k_1+k_2=k}\frac{2k_1k_2}{k^2}\int_0^{\tau}{\rm e}^{-2is(k^2-k_1k_2)}ds \widehat{\overline{f}}_{k_1}\widehat{\overline{f}}_{k_2}{\rm e}^{ikx}\notag\\
	&=T_1^{\tau}(f)+T_2^{\tau}(f)+T_3^{\tau}(f).
\end{align}
By symmetry, obviously we have $T_1^{\tau}(f)=T_2^{\tau}(f)$ and it suffices to approximate $T_1^{\tau}(f)$ and $T_3^{\tau}(f)$. To begin with, we decompose $T_1^{\tau}(f)$ as
\begin{align}\label{t1}
T_1^{\tau}(f)&=\sum\limits_{k\not=0}\sum\limits_{k_1+k_2=k}\frac{k_1^2}{k^2}\int_0^{\tau}{\rm e}^{-2is(k_2^2+kk_1)}ds \widehat{\overline{f}}_{k_1}\widehat{\overline{f}}_{k_2}{\rm e}^{ikx}\notag\\
	&=\sum\limits_{k\not=0}\sum\limits_{k_1+k_2=k}\frac{k_1^2}{k^2}\int_0^{\tau}{\rm e}^{-2isk_2^2}ds \widehat{\overline{f}}_{k_1}\widehat{\overline{f}}_{k_2}{\rm e}^{ikx}\notag\\
&\quad+\sum\limits_{k\not=0}\sum\limits_{k_1+k_2=k}\frac{k_1^2}{k^2}\int_0^{\tau}\left({\rm e}^{-2iskk_1}-1\right)ds \widehat{\overline{f}}_{k_1}\widehat{\overline{f}}_{k_2}{\rm e}^{ikx}\notag\\	&\quad+\sum\limits_{k\not=0}\sum\limits_{k_1+k_2=k}\frac{k_1^2}{k^2}\int_0^{\tau}\big({\rm e}^{-2isk_2^2}-1\big)\left({\rm e}^{-2iskk_1}-1\right)ds \widehat{\overline{f}}_{k_1}\widehat{\overline{f}}_{k_2}{\rm e}^{ikx}\notag\\
	&=L_1^{\tau}(f)+L_2^{\tau}(f)+P_1^{\tau}(f).
\end{align}
For $f$ with $\mathcal{F}_0(f)=0$, similarly $L_1^{\tau}(f)$ and $L_2^{\tau}(f)$ can be integrated exactly as
\begin{align}
	L_1^{\tau}(f)&=\sum\limits_{k\not=0}\sum\limits_{k_1+k_2=k \atop k_2=0}\frac{k_1^2}{k^2}\int_0^{\tau}{\rm e}^{-2isk_2^2}ds \widehat{\overline{f}}_{k_1}\widehat{\overline{f}}_{k_2}{\rm e}^{ikx}\notag\\
	&\quad+\sum\limits_{k\not=0}\sum\limits_{k_1+k_2=k \atop k_2\not=0}\frac{k_1^2}{k^2}\int_0^{\tau}{\rm e}^{-2isk_2^2}ds \widehat{\overline{f}}_{k_1}\widehat{\overline{f}}_{k_2}{\rm e}^{ikx}\notag\\
	&=-\frac{i}{2}\pa_x^{-2}\left[\left({\rm e}^{2i\tau \pa_x^2}\pa_x^{-2}\ov{f}\right)\left(\pa_x^2\ov{f}\right)\right]+\frac{i}{2}\pa_x^{-2}
\left[\left(\pa_x^{2}\ov{f}\right)\left(\pa_x^{-2}\ov{f}\right)\right],\label{i1}\\
	L_2^{\tau}(f)&=-\frac{i}{2}{\rm e}^{i\tau \pa_x^2}\pa_x^{-3}\left[\left({\rm e}^{i\tau \pa_x^2}\pa_x^{}\ov{f}\right)\left({\rm e}^{-i\tau \pa_x^2}\ov{f}\right)\right]+\frac{i}{2}\pa_x^{-3}\left[\left(\pa_x^{}\ov{f}\right)\ov{f}\right]\notag\\ &\quad-\tau\pa_x^{-2}\left[\left(\pa_x^{2}\ov{f}\right)\ov{f}\right].\label{i2}
\end{align}
The remainder term $P_1^\tau(f)$ will be thrown away in the scheme and the estimate is postponed to the next section.
Similarly $T_3^{\tau}(f)$ can be decomposed as
\begin{align}\label{t3}
T_3^{\tau}(f)&=\sum\limits_{k\not=0}\sum\limits_{k_1+k_2=k}\frac{2k_1k_2}{k^2}\int_0^{\tau}{\rm e}^{-2is(k^2-k_1k_2)}ds \widehat{\overline{f}}_{k_1}\widehat{\overline{f}}_{k_2}{\rm e}^{ikx}\notag\\
	&=\sum\limits_{k\not=0}\sum\limits_{k_1+k_2=k}\frac{2k_1k_2}{k^2}\int_0^{\tau}{\rm e}^{-2isk^2}ds \widehat{\overline{f}}_{k_1}\widehat{\overline{f}}_{k_2}{\rm e}^{ikx}\notag\\
&\quad+\sum\limits_{k\not=0}\sum\limits_{k_1+k_2=k}\frac{2k_1k_2}{k^2}\int_0^{\tau}\left({\rm e}^{2isk_1k_2}-1\right)ds \widehat{\overline{f}}_{k_1}\widehat{\overline{f}}_{k_2}{\rm e}^{ikx}\notag\\
&\quad+\sum\limits_{k\not=0}\sum\limits_{k_1+k_2=k}\frac{2k_1k_2}{k^2}\int_0^{\tau}\big({\rm e}^{-2isk^2}-1\big)\left({\rm e}^{2isk_1k_2}-1\right)ds \widehat{\overline{f}}_{k_1}\widehat{\overline{f}}_{k_2}{\rm e}^{ikx}\notag\\
	&=L_3^{\tau}(f)+L_4^{\tau}(f)+P_2^{\tau}(f),
\end{align}
where $L_3^{\tau}(f)$ and $L_4^{\tau}(f)$ can be integrated exactly as
\begin{align}
	L_3^{\tau}(f)&=-i\pa_x^{-4}\left({\rm e}^{2i\tau \pa_x^2}-1\right)\left(\pa_x\ov{f}\right)^2, \label{i3}\\
	L_4^{\tau}(f)&=i\pa_x^{-2}{\rm e}^{-i\tau \pa_x^2}\left({\rm e}^{i\tau \pa_x^2}\ov{f}\right)^2-i\pa_x^{-2}\left(\ov{f}\right)^2-2\tau\pa_x^{-2}
\left(\pa_x\ov{f}\right)^2.\label{i4}
\end{align}

Combining \eqref{uapp}, \eqref{l1f}, \eqref{l2f}, \eqref{t1} and \eqref{t3}, we obtain
\begin{align}\label{firstorderapp}
u(t_n+\tau)=\Psi_1^{\tau}(u(t_n))-\frac{i}{2}B^{\tau}P_1^\tau(u(t_n))-\frac{i}{4}B^{\tau}P_2^\tau(u(t_n))+\mathcal{R}_{0}(\tau^2),
\end{align}
where
\begin{align}
	\Psi_1^{\tau}(f)&={\rm e}^{i\tau \langle \partial_x^2 \rangle}f-\frac{i}{4}B^{\tau}\big[T_0^{\tau}(f)+2L_1^{\tau}(f)+2L_2^{\tau}(f)+L_3^{\tau}
(f)+L_4^{\tau}(f)+I_1^{\tau}(f)\notag\\
	&\quad +2I_2^{\tau}(f)\big]-i\tau (at_n+b)B^{\tau}\big(f +\psi_1(2i\tau \partial_x^2)\overline{f}\big),\label{sch1}
\end{align}
with operators $B^{\tau}$, $I_1^{\tau}$, $I_2^{\tau}$, $T_0^{\tau}$, $L_1^{\tau}$, $L_2^{\tau}$, $L_3^{\tau}$, $L_4^{\tau}$ defined in \eqref{BL0}, \eqref{l1f}, \eqref{l2f}, \eqref{t0}, \eqref{i1}, \eqref{i2}, \eqref{i3}, \eqref{i4} respectively. Furthermore, noticing $B^{\tau}T_0^{\tau}(f)=0$, we can rewrite \eqref{sch1} simply as
\begin{align}\label{sch2}
	\Psi_1^{\tau}(f)&={\rm e}^{i\tau \langle \partial_x^2 \rangle}f-\frac{i}{4}B^{\tau}\big[2L_1^{\tau}(f)+2L_2^{\tau}(f)+L_3^{\tau}(f)+L_4^{\tau}(f)\notag\\
	&\quad +I_1^{\tau}(f)+2I_2^{\tau}(f)\big]-i\tau (at_n+b)B^{\tau}\big(f +\psi_1(2i\tau \partial_x^2)\overline{f}\big),
\end{align}
which is exactly \eqref{sch0} when $I^\tau_j$ and $L_j^\tau$ are plugged in.
Recalling \eqref{uvz} and $z=\mathcal{F}_0(z)+\check{z}$, now we are able to propose the scheme
\be\label{znscheme}
z^n=\frac{1}{2}(u^n+\overline{u^n})+at_n+b,\quad
	z^n_t=\frac{i}{2}\langle\partial_x^2\rangle(u^n-\overline{u^n})+a,
\ee
where
\be\label{firstsch}
u^{n+1}=\Psi_1^{\tau}(u^{n}),\quad n\ge 0, \quad u^0=u(0, x).
\ee
The proposed scheme is fully explicit in time and it is easy to implement efficiently if pseudospectral method is used for spatial discretization thanks to FFT.

\section{Error estimates}\label{4}
In this section, we will establish the global error estimate concerning the first-order scheme \eqref{sch2}.

\subsection{Local error estimate}
In this part, we give the local error estimate of the scheme \eqref{sch2}. Inspired by \eqref{firstorderapp} and \eqref{l241}, it remains to estimate $P_1^{\tau}(f)$  and $P_2^\tau(f)$.

\begin{lemma}\label{P1es}
For $r\ge 1$, $f\in H^{r+p(r)}$, it holds
\[\|P_1^\tau(f)\|_r\lesssim \tau^2\|f\|_{r+p(r)}^2.\]
\end{lemma}
\begin{proof}
Firstly the $k$-th Fourier coefficient of $P_1^{\tau}(f)$ can be bounded as
\begin{align}\label{local1}
	\lvert\mathcal{F}_k&\left(P_1^\tau(f)\right)\rvert\notag\\
	&\lesssim \sum\limits_{k_1+k_2=k}\lvert k\rvert^{-2} \lvert k_1\rvert^{2} \left\lvert\int_0^{\tau}\big({\rm e}^{-2isk_2^2}-1\big)\left({\rm e}^{-2iskk_1}-1\right)ds\right\rvert \left\lvert \widehat{\overline{f}}_{k_1}\right\rvert \left\lvert \widehat{\overline{f}}_{k_2}\right\rvert\notag\\
	&\lesssim \tau \sum\limits_{k_1+k_2=k}\lvert k\rvert^{-2} \lvert k_1\rvert^{2} \sup\limits_{0\le s \le \tau}\big(\big\lvert\big({\rm e}^{-2isk_2^2}-1\big)\big\rvert \left\lvert\left({\rm e}^{-2iskk_1}-1\right)\right\rvert\big) \left\lvert \widehat{\overline{f}}_{k_1}\right\rvert \left\lvert \widehat{\overline{f}}_{k_2}\right\rvert\notag\\
	&\lesssim \tau^{1+\alpha+\beta} \sum\limits_{k_1+k_2=k}\lvert k\rvert^{-2} \lvert k_1\rvert^{2} \left\lvert k_2\right\rvert^{2\alpha}\left\lvert k \right\rvert^{\beta}\left\lvert k_1 \right\rvert^{\beta} \left\lvert \widehat{\overline{f}}_{k_1}\right\rvert \left\lvert \widehat{\overline{f}}_{k_2}\right\rvert\notag\\
	&\lesssim \tau^{1+\alpha+\beta} \lvert k\rvert^{-2+\beta}\sum\limits_{k_1+k_2=k} \lvert k_1\rvert^{2+\beta} \left\lvert k_2\right\rvert^{2\alpha} \left\lvert \widehat{\overline{f}}_{k_1}\right\rvert \left\lvert \widehat{\overline{f}}_{k_2}\right\rvert,
\end{align}
where $\alpha$, $\beta$ $\in$ $[0,1]$. This gives
\begin{align}
\left\|P_1^\tau(f)\right\|_r&\lesssim \Big\|\tau^{1+\alpha+\beta} \sum\limits_{k\not=0} \lvert k\rvert^{-2+\beta}\sum\limits_{k_1+k_2=k} \lvert k_1\rvert^{2+\beta} \left\lvert k_2\right\rvert^{2\alpha} \left\lvert \widehat{\overline{f}}_{k_1}\right\rvert \left\lvert \widehat{\overline{f}}_{k_2}\right\rvert{\rm e}^{ikx}\Big\|_{r}\notag\\
	&\lesssim \tau^{1+\alpha+\beta}\left\| \lvert\pa_x\rvert^{-2+\beta}\left[\left( \left\lvert \pa_x\right\rvert^{2+\beta}\widetilde{\ov{f}}\right) \left( \left\lvert \pa_x\right\rvert^{2\alpha}\widetilde{\ov{f}}\right)\right]\right\|_{r}.\label{P1e}
\end{align}

Now we give several estimates for $P_1^\tau(f)$ which might be valid in different regimes.

(1) For $r\ge 1$, setting $\alpha=1$ and $\beta=0$ in \eqref{P1e}, applying the inequalities in Lemma \ref{newb} yields
\be\label{locallemma5}
	\|P_1^\tau(f)\|_r\lesssim \tau^{2}\left\| \lvert\pa_x\rvert^{-2}\big[\big( \left\lvert \pa_x\right\rvert^{2}\widetilde{\ov{f}}\big) \big( \left\lvert \pa_x\right\rvert^{2}\widetilde{\ov{f}}\big)\big]\right\|_{r}\lesssim \tau^{2}\big\| \left\lvert \pa_x\right\rvert^{}\widetilde{\ov{f}}\big\|_{r}^2=\tau^2\|f\|^2_{r+1},
\ee
which implies a second-order local error by requiring one additional derivative.

(2) By applying Lemma \ref{lemma3} (iv), for $r+\beta-1>1/2$, we get
\begin{align*}
	\left\|P_1^\tau(f)\right\|_r&\lesssim \tau^{1+\alpha+\beta}\left\| \lvert\pa_x\rvert^{-1}\left[\left( \left\lvert \pa_x\right\rvert^{2+\beta}\widetilde{\ov{f}}\right) \left( \left\lvert \pa_x\right\rvert^{2\alpha}\widetilde{\ov{f}}\right)\right]\right\|_{r+\beta-1}\notag\\
	&\lesssim \tau^{1+\alpha+\beta}\left\| \left( \left\lvert \pa_x\right\rvert^{1+\beta}\widetilde{\ov{f}}\right)\right\|_{r+\beta-1} \left\| \left( \left\lvert \pa_x\right\rvert^{2\alpha}\widetilde{\ov{f}}\right)\right\|_{r+\beta-1}\notag\\
	&\lesssim \tau^{1+\alpha+\beta}\|f\|_{r+2\beta}\|f\|_{r+2\alpha+\beta-1}.
\end{align*}
To get a local error bound of order two, setting $\alpha+\beta=1$, one obtains
\begin{align}\label{locallemma3}
	\left\|P_1^\tau(f)\right\|_r\lesssim \tau^{2}\|f\|_{r+2\beta}^2, \quad \text{with} \quad\beta\in\left[1/3, 1/2\right],\quad r>3/2-\beta.
\end{align}
We clearly see that compared to the estimate in (1), this decreases the additional regularity required when $r>1$ and the least order of additional regularity can be decreased to $2/3$ which is valid when $r>7/6$. On the other hand, w observe that it is possible to require less additional regularity by choosing smaller $\beta$, however, we have to pay extra price that the error itself is estimated in a much more regular space by noticing the constraint $r>3/2-\beta$.

(3) When $r+\beta-2\in [0, 1/2)$, by applying Lemma \ref{sobolevcompact} and H{\" o}lder inequality, one gets
\begin{align*}
	\left\|P_1^\tau(f)\right\|_r&\lesssim \tau^{1+\alpha+\beta}\left\| \left( \left\lvert \pa_x\right\rvert^{2+\beta}\widetilde{\ov{f}}\right) \left( \left\lvert \pa_x\right\rvert^{2\alpha}\widetilde{\ov{f}}\right)\right\|_{r+\beta-2}\notag\\
	&\lesssim \tau^{1+\alpha+\beta}\left\| \left( \left\lvert \pa_x\right\rvert^{2+\beta}\widetilde{\ov{f}}\right) \left( \left\lvert \pa_x\right\rvert^{2\alpha}\widetilde{\ov{f}}\right)\right\|_{L^{\frac{2}{5-2r-2\beta}}}\notag\\
	&\lesssim \tau^{1+\alpha+\beta}\left\| \left( \left\lvert \pa_x\right\rvert^{2+\beta}\widetilde{\ov{f}}\right)\right\|_{L^{\frac{4}{5-2r-2\beta}}} \left\| \left( \left\lvert \pa_x\right\rvert^{2\alpha}\widetilde{\ov{f}}\right)\right\|_{L^{\frac{4}{5-2r-2\beta}}}\notag\\
	&\lesssim \tau^{1+\alpha+\beta}\left\| f\right\|_{\frac{2r+2\beta-3}{4}+2+\beta} \left\| f\right\|_{\frac{2r+2\beta-3}{4}+2\alpha}\notag\\
	&\lesssim \tau^{1+\alpha+\beta}\left\| f\right\|_{\frac{r}{2}+\frac{3}{2}\beta+\frac{5}{4}}^2.
\end{align*}
Similarly setting $\alpha+\beta=1$, one derives
\be\label{localcompact}
\left\|P_1^\tau(f)\right\|_r\lesssim \tau^{2}\left\| f\right\|_{\frac{r}{2}+\frac{3}{2}\beta+\frac{5}{4}}^2, \quad r\in (3/2-\beta, 2-\beta], \quad \beta \in [0,1].
\ee

(4) On the other hand, we can estimate $P_1^\tau(f)$ by employing the inequality \eqref{kp1} in Lemma \ref{lemma1}, by setting $\beta=0$, $\alpha=1$ in \eqref{local1},
\begin{align}\label{local3}
	\left\|P_1^\tau(f)\right\|_r&\lesssim \tau^{2}\left\| J^{r-2}\left( \left\lvert \pa_x\right\rvert^{2}\widetilde{\ov{f}}\right) \left( \left\lvert \pa_x\right\rvert^{2}\widetilde{\ov{f}}\right)\right\|_{}\notag\\
	&\lesssim \tau^{2}\left\| \left( J^r\widetilde{\ov{f}}\right)\right\|_{L^{p_1}} \left\| \left( \left\lvert \pa_x\right\rvert^{2}\widetilde{\ov{f}}\right)\right\|_{L^{p_2}},
\end{align}
where $2\le p_1<\infty$, $2< p_2\le \infty$ and $\frac{1}{p_1}+\frac{1}{p_2}=\frac{1}{2}$, when $r>2$. Applying the Sobolev embedding theorem in Lemma \ref{sobolevcompact}, we get
\[\left\|P_1^\tau(f)\right\|_r\lesssim \tau^{2}\|f\|_{r-\frac{1}{p_1}+\frac{1}{2}}\|f\|_{\frac{5}{2}-\frac{1}{p_2}},\quad 2<p_1, p_2<\infty.\]
To obtain a lower spatial regularity requirement, it is natural to choose $r-\frac{1}{p_1}+\frac{1}{2}=\frac{5}{2}-\frac{1}{p_2}$ for $\frac{1}{p_1}+\frac{1}{p_2}=\frac{1}{2}$ with $p_1\in(2, \infty)$ and $p_2\in(2, \infty)$, i.e., $p_1=\frac{1}{\frac{r}{2}-\frac{3}{4}}$ and $p_2=\frac{1}{\frac{5}{4}-\frac{r}{2}}$ with $r\in(2, 5/2)$, which yields the local error estimate as
\be\label{pes1}
\lVert P_1^{\tau}(f)\rVert_r\lesssim \tau^{2}\|f\|_{\frac{r}{2}+\frac{5}{4}}^2,\quad  r\in(2, 5/2).
\ee

(5) Finally, using the bilinear inequality \eqref{kp} in Lemma \ref{lemma1}, for $\alpha+\beta=1$, one has
\begin{align}\label{localp1}
	\left\|P_1^\tau(f)\right\|_r&\lesssim \tau^{1+\alpha+\beta}\left\| \left( \left\lvert \pa_x\right\rvert^{2+\beta}\widetilde{\ov{f}}\right) \left( \left\lvert \pa_x\right\rvert^{2\alpha}\widetilde{\ov{f}}\right)\right\|_{r+\beta-2}\notag\\
	&\lesssim \tau^{2} \left\| \left( \left\lvert \pa_x\right\rvert^{2+\beta}\widetilde{\ov{f}}\right)\right\|_{r+\beta-2} \left\| \left( \left\lvert \pa_x\right\rvert^{2\alpha}\widetilde{\ov{f}}\right)\right\|_{r+\beta-2}\notag\\
	&\lesssim \tau^{2}\|f\|_{r+2\beta}\|f\|_{r+2\alpha+\beta-2}\notag\\
&\lesssim \tau^{2}\|f\|^2_{r+2\beta},
\end{align}
for $r>5/2-\beta$ with $0\le\beta\le 1$. This implies an error without loss of regularity when $r>5/2$ by choosing $\beta=0$.

Taking $\beta=0$ in \eqref{localcompact}, one gets
\[\left\|P_1^\tau(f)\right\|_r\lesssim \tau^{2}\left\| f\right\|_{\frac{r}{2}+\frac{5}{4}}^2, \quad r\in (3/2, 2],\]
which combines with \eqref{pes1} and \eqref{localp1} gives
\be\label{32}
\begin{split}
&\left\|P_1^\tau(f)\right\|_r\lesssim \tau^{2}\left\| f\right\|_{\frac{r}{2}+\frac{5}{4}}^2, \quad r\in (3/2, 5/2);\\
&\left\|P_1^\tau(f)\right\|_r\lesssim \tau^{2}\left\| f\right\|_{r+}^2, \quad r=5/2;\\
&\left\|P_1^\tau(f)\right\|_r\lesssim \tau^{2}\left\| f\right\|_{r}^2, \quad r> 5/2.
\end{split}
\ee
For $r\le \frac32$, by \eqref{localcompact}, we have to choose $\beta=(\frac32-r)+$, i.e., $\beta=\frac32-r+\varepsilon$ with any sufficiently small $\varepsilon>0$ which reads as
\be\label{32l}
\left\|P_1^\tau(f)\right\|_r\lesssim \tau^{2}\left\| f\right\|_{\frac{r}{2}+\frac{5}{4}+\frac{3}{2}(\frac{3}{2}-r)+}^2=\left\| f\right\|_{(\frac{7}{2}-r)+}^2, \quad r \in [1, 3/2].\ee
Similarly \eqref{locallemma3} equivalents to the estimate
\be\label{32r}
\left\|P_1^\tau(f)\right\|_r\lesssim \tau^{2}\left\| f\right\|_{(3-r)+}^2,\quad r\in (1, \frac76];\quad
\left\|P_1^\tau(f)\right\|_r\lesssim \tau^{2}\left\| f\right\|_{r+\frac23}^2, \quad r>\frac76.\ee
Lemma \ref{P1es} is concluded by taking the minimum of the required order of regularity for \eqref{32}, \eqref{32l}, \eqref{32r} and \eqref{locallemma3}.
\end{proof}

Concerning the term $P_2^{\tau}(f)$, we have the following estimate.

\begin{lemma}\label{P2es}
For $r\ge 1$, we have
\[\|P_2^\tau(f)\|_r\lesssim \tau^2\|f\|_{r+q(r)}^2,\]
where
\[q(r)=
\left\{
\begin{aligned}
&5/4-r/2,\quad &1\le r\le 5/2;\\
&0,\quad &r>5/2.
\end{aligned}\right.\]
\end{lemma}
\begin{proof}
It follows from \eqref{t3} that
\begin{align}\label{local2}
	\lvert&\mathcal{F}_k\left(P_2^\tau(f)\right)\rvert\notag\\
	&\lesssim \sum\limits_{k_1+k_2=k}\lvert k\rvert^{-2} \lvert k_1\rvert^{}\lvert k_2\rvert^{} \left\lvert\int_0^{\tau}\big({\rm e}^{-2isk^2}-1\big)\left({\rm e}^{2isk_1k_2}-1\right)ds\right\rvert \left\lvert \widehat{\overline{f}}_{k_1}\right\rvert \left\lvert \widehat{\overline{f}}_{k_2}\right\rvert\notag\\
	&\lesssim \tau \sum\limits_{k_1+k_2=k}\lvert k\rvert^{-2} \lvert k_1\rvert^{}\lvert k_2\rvert^{} \sup\limits_{0\le s\le \tau}\big(\big\lvert\big({\rm e}^{-2isk^2}-1\big)\big\rvert \left\lvert\left({\rm e}^{2isk_1k_2}-1\right)\right\rvert\big) \left\lvert \widehat{\overline{f}}_{k_1}\right\rvert \left\lvert \widehat{\overline{f}}_{k_2}\right\rvert\notag\\
	&\lesssim \tau^{1+\alpha+\beta} \sum\limits_{k_1+k_2=k}\lvert k\rvert^{-2} \lvert k_1\rvert^{}\lvert k_2\rvert^{} \left\lvert k \right\rvert^{2\alpha}\left\lvert k_1 \right\rvert^{\beta} \left\lvert k_2 \right\rvert^{\beta}\left\lvert \widehat{\overline{f}}_{k_1}\right\rvert \left\lvert \widehat{\overline{f}}_{k_2}\right\rvert\notag\\
	&\lesssim \tau^{2} \lvert k\rvert^{-2+2\alpha}\sum\limits_{k_1+k_2=k} \lvert k_1\rvert^{1+\beta} \left\lvert k_2\right\rvert^{1+\beta} \left\lvert \widehat{\overline{f}}_{k_1}\right\rvert \left\lvert \widehat{\overline{f}}_{k_2}\right\rvert,
\end{align}
for $\alpha, \beta\in[0,1]$ satisfying $\alpha+\beta=1$.
Using similar approach applied in the proof of Lemma \ref{P1es}, we establish several estimates by applying various tools.

(1) When $r+2\alpha-2\in (-1/2, 0]$, by applying Lemma \ref{sobolevcompact} and H{\" o}lder inequality, one gets
\begin{align*}
	\left\|P_2^\tau(f)\right\|_r&\lesssim \tau^{2} \left\| \left( \left\lvert \pa_x\right\rvert^{1+\beta}\widetilde{\ov{f}}\right) \left( \left\lvert \pa_x\right\rvert^{1+\beta}\widetilde{\ov{f}}\right)\right\|_{r+2\alpha-2}\notag\\
	&\lesssim \tau^{2} \left\| \left( \left\lvert \pa_x\right\rvert^{1+\beta}\widetilde{\ov{f}}\right) \left( \left\lvert \pa_x\right\rvert^{1+\beta}\widetilde{\ov{f}}\right)\right\|_{L^{\frac{2}{5-2r-4\alpha}}}\notag\\
	&\lesssim\tau^{2}  \left\| \left( \left\lvert \pa_x\right\rvert^{1+\beta}\widetilde{\ov{f}}\right)\right\|_{L^{\frac{4}{5-2r-4\alpha}}}^2\notag\\
	&\lesssim\tau^{2}  \left\| f\right\|_{\frac{r}{2}+\frac{5}{4}}^2.
\end{align*}
Noticing the constraint $r\in (3/2-2\alpha, 2-2\alpha]$ and $\alpha\in [0, 1]$, we immediately get
\be\label{localcompactp2}
\left\|P_2^\tau(f)\right\|_r\lesssim\tau^{2}  \left\| f\right\|_{\frac{r}{2}+\frac{5}{4}}^2,\quad  r\in [1, 2].
\ee

(2) Setting $\alpha=0$ and $\beta=1$, one gets
\[\left\|P_2^\tau(f)\right\|_r\lesssim \tau^{2}\left\| |\pa_x|^{-2}\left( \left\lvert \pa_x\right\rvert^{2}\widetilde{\ov{f}}\right) \left( \left\lvert \pa_x\right\rvert^{2}\widetilde{\ov{f}}\right)\right\|_{r}\lesssim \tau^{2}\left\| J^{r-2}\left( \left\lvert \pa_x\right\rvert^{2}\widetilde{\ov{f}}\right) \left( \left\lvert \pa_x\right\rvert^{2}\widetilde{\ov{f}}\right)\right\|,\]
which is exactly the same as in \eqref{local3}. Hence \eqref{pes1} also holds for $P_2^\tau(f)$, i.e.,
\be\label{p2f}
\lVert P_2^{\tau}(f)\rVert_r\lesssim \tau^{2}\|f\|_{\frac{r}{2}+\frac{5}{4}}^2,\quad  r\in(2, 5/2].
\ee

(3) It remains to give a bound of $\|P_2^\tau(f)\|_r$ for $r>5/2$. Employing the bilinear estimate \eqref{kp}, one easily gets
\begin{align*}
	\left\|P_2^\tau(f)\right\|_r&\lesssim \tau^{2} \left\| \left( \left\lvert \pa_x\right\rvert^{1+\beta}\widetilde{\ov{f}}\right) \left( \left\lvert \pa_x\right\rvert^{1+\beta}\widetilde{\ov{f}}\right)\right\|_{r+2\alpha-2}\notag\\
	&\lesssim\tau^{2}  \left\| \left( \left\lvert \pa_x\right\rvert^{1+\beta}\widetilde{\ov{f}}\right)\right\|_{r+2\alpha-2}^2\notag\\
	&\lesssim \tau^{2}  \|f\|^2_{r+\alpha},\quad \mathrm{for}\quad r>5/2-2\alpha,
\end{align*}
which implies
\[\left\|P_2^\tau(f)\right\|_r\lesssim \tau^{2}\|f\|^2_{r},\quad \mathrm{for}\quad r>5/2;\quad \left\|P_2^\tau(f)\right\|_r\lesssim \tau^{2}\|f\|^2_{(\frac{r}{2}+\frac{5}{4})+},\quad r\in [1, 5/2].\]
This together with \eqref{localcompactp2} and \eqref{p2f} concludes Lemma \ref{P2es}.
\end{proof}

It is easy to see $q(r)\le p(r)$ by direct computations. In spirit of \eqref{firstorderapp}, Lemma \ref{P1es} and Lemma \ref{P2es}, we obtain the local error of the scheme \eqref{sch2}.

\begin{lemma}\label{localerror}
Suppose $r\ge 1$, $u\in L^\infty(0, T; H^{r+p(r)})$. Then we have
\[\|u(t_n+\tau)-\Psi_1^{\tau}(u(t_n))\|_r\le L \tau^2,\]
where $L$ depends on $\|u\|_{L^\infty(0, T; H^{r+p(r)})}$.
\end{lemma}

\subsection{Stability}
\begin{lemma}\label{sta}
	Suppose $r\ge 1$ and $f,g$ $\in$ $H^r$. Then for $\tau> 0$, we have
	\be\label{stab1}
	\lVert \Psi_1^{\tau} (f)-\Psi_1^{\tau} (g)\rVert_r \leq  (1+M\tau)\lVert f-g\rVert_r,
	\ee
	where $M$ depends on $r$ and $\lVert f\rVert_r+\lVert g\rVert_r$.
\end{lemma}
\begin{proof}
	To begin with, by using \eqref{kp}, \eqref{BL0} and \eqref{L12}, one can easily check that
	\begin{align}\label{stability1}
		\|I_{1}^{\tau}(f)-I_{1}^{\tau}(g)\|_r &\le C \tau \mathop{{\rm sup}}\limits_{\tau\ge s\ge 0}\left\|e^{is\pa_x^2}\left[(e^{-is\pa_x^2}f)^2-(e^{-is\pa_x^2}g)^2\right]\right\|_r\notag\\
		&\le C  \tau \mathop{{\rm sup}}\limits_{\tau\ge s\ge 0}\|e^{-is\pa_x^2}(f+g)\|_r\|e^{-is\pa_x^2}(f-g)\|_r\notag\\
		&\le C\tau (\|f\|_r+\|g\|_r)\|f-g\|_r.
	\end{align}
	Similar discussions for $I_1^{\tau}(f)$ and $I_2^{\tau}(f)$ yield that
	\begin{align}\label{stability2}
		\|I_{2}^{\tau}(f)-I_{2}^{\tau}(g)\|_r&\le C \tau(\|f\|_r+\|g\|_r)\|f-g\|_r,\notag\\
		\|I_{0}^{\tau}(f)-I_{0}^{\tau}(g)\|_r&\le C \tau(\|f\|_r+\|g\|_r)\|f-g\|_r.
	\end{align}
	Noticing the decomposition of $T_i^{\tau}(f)$ ($i=1,2,3$) in \eqref{t0}, \eqref{t1} and \eqref{t3}, we have
	\begin{align}\label{stability3}
		I_0^{\tau}(f)=T_0^{\tau}(f)+2L_1^{\tau}(f)+2L_2^{\tau}(f)+2P_1^{\tau}(f)+L_3^{\tau}(f)+L_4^{\tau}(f)+P_2^{\tau}(f),
	\end{align}
	which yields
	\begin{align}\label{stability4}
		\|W^\tau&(f)-W^\tau(g)\|_r= \|I_0^{\tau}(f)-I_0^{\tau}(g)-2P_1^{\tau}(f)+2P_1^{\tau}(g)-P_2^{\tau}(f)+P_2^{\tau}(g)\|_r\notag\\
		&\le\|I_0^{\tau}(f)-I_0^{\tau}(g)\|_r+2\|P_1^{\tau}(f)-P_1^{\tau}(g)\|_r+\|P_2^{\tau}(f)-P_2^{\tau}(g)\|_r,
	\end{align}
	where $W^{\tau}(f):=T_0^{\tau}(f)+2L_1^{\tau}(f)+2L_2^{\tau}(f)+L_3^{\tau}(f)+L_4^{\tau}(f)$.
	
	It remains to deal with the terms $P_1^{\tau}(f)$ and $P_2^{\tau}(f)$. According to Lemmas \ref{lemma1}, \ref{lemma3}, \ref{newb} and the definition of $P_1^{\tau}$ in \eqref{t1}, for $r\ge 1$, we have
	\begin{align}\label{stability5}
		&\|P_1^{\tau}(f)-P_1^{\tau}(g)\|_r\notag\\
		&\lesssim		\Big\|\sum\limits_{k\not=0}\sum\limits_{k_1+k_2=k}\frac{k_1^2}{k^2}\int_0^{\tau}\big({\rm e}^{-2isk_2^2}-1\big)\left({\rm e}^{-2iskk_1}-1\right)ds \big(\widehat{\overline{f}}_{k_1}\widehat{\overline{f}}_{k_2}-\widehat{\overline{g}}_{k_1}\widehat{\overline{g}}_{k_2}\big){\rm e}^{ikx}\Big\|_r\notag\\
		&\lesssim
		\tau\Big\|\sum\limits_{k\not=0}\sum\limits_{k_1+k_2=k}k^{-2}k_1^2 \left\lv\widehat{\overline{f}}_{k_1}\widehat{\overline{f}}_{k_2}-\widehat{\overline{f}}_{k_1}\widehat{\overline{g}}_{k_2}+\widehat{\overline{f}}_{k_1}\widehat{\overline{g}}_{k_2}-\widehat{\overline{g}}_{k_1}\widehat{\overline{g}}_{k_2}\right\rv{\rm e}^{ikx}\Big\|_r\notag\\
		&\lesssim
		\tau\Big\|\sum\limits_{k\not=0}\sum\limits_{k_1+k_2=k}k^{-2}k_1^2 \left\lv\widehat{\overline{f}}_{k_1}\right\rv\left\lv\widehat{\overline{f}}_{k_2}-\widehat{\overline{g}}_{k_2}\right\rv+\left\lv\widehat{\overline{f}}_{k_1}-\widehat{\overline{g}}_{k_1}\right\rv\left\lv\widehat{\overline{g}}_{k_2}\right\rv{\rm e}^{ikx}\Big\|_r\notag\\
		&\lesssim
		\tau\left(\left\|\pa_x^{-2}\left[(\pa_x^2\widetilde{\ov{f}})(\widetilde{\ov{f}-\ov{g}})\right]\right\|_r+\left\|\pa_x^{-2}\left[\pa_x^2(\widetilde{\ov{f}-\ov{g}})\widetilde{\ov{g}}\right]\right\|_r\right)\notag\\
		&=	\tau\Big(\left\|\pa_x^{-1}\left[(\pa_x\widetilde{\ov{f}})(\widetilde{\ov{f}-\ov{g}})\right]-\pa_x^{-2}\left[(\pa_x\widetilde{\ov{f}})\pa_x(\widetilde{\ov{f}-\ov{g}})\right]\right\|_r\notag\\
		&\quad +\left\|\pa_x^{-1}\left[\pa_x(\widetilde{\ov{f}-\ov{g}}) \widetilde{\ov{g}}\right]-\pa_x^{-2}\left[\pa_x(\widetilde{\ov{f}-\ov{g}}) \pa_x\widetilde{\ov{g}}\right]\right\|_r\Big)\notag\\
		&\lesssim \tau (\|f\|_r+\|g\|_r)\|f-g\|_r,
	\end{align}
	where we have used the modified version of Newton-Leibniz formula \begin{align*}
		\pa_x^{-2}[(\pa_x^2f)g]=\pa_x^{-1}[(\pa_xf)g]-\pa_x^{-2}[(\pa_xf)( \pa_xg)],
	\end{align*}
	which can be obviously derived by the decomposition
	\[k^{-2}k_1^2=k^{-2}k_1(k-k_2)=k^{-1}k_1-k^{-2}k_1k_2.\]
Recalling the definition of $P_2^{\tau}(f)$ in \eqref{t3}, we can establish
	\begin{align}\label{stability6}
		\|P_2^{\tau}(f)-P_2^{\tau}(g)\|_r\lesssim \tau (\|f\|_r+\|g\|_r)\|f-g\|_r,
	\end{align}
by employing similar arguments as above.
	Combining \eqref{stability1}--\eqref{stability6}, applying Lemma \ref{lemma1} and Lemma \ref{lemma3}, for $r\ge 1$, we have
	\begin{align}\label{stability7}
		\|&\Psi_1^{\tau}(f)-\Psi_1^{\tau} (g)\|_r=\bigg\| {\rm e}^{i\tau \la \pa_x^2\ra}(f-g)-\frac{i}{4}B^{\tau}\big[W^{\tau}(f)-W^{\tau}(g)+I_{1}^{\tau}(f)-I_{1}^{\tau}(g)\notag\\
		&\qquad+2\left(I_{2}^{\tau}(f)-I_{2}^{\tau}(g)\right)\big]-i\tau (at_n+b)B^{\tau}\big(f-g +\psi_1(2i\tau \partial_x^2)(\overline{f}-\ov{g})\big)  \bigg\|_r\notag\\
		&\le \|f-g\|_r+\Big[\big\|W^{\tau}(f)-W^{\tau}(g)\big\|_r+\left\|I_{1}^{\tau}(f)-I_{1}^{\tau}(g)\right\|_r\notag\\
&\quad+2\left\|I_{2}^{\tau}(f)-I_{2}^{\tau}(g)\right\|_r\Big]+C_r\tau(\|f-g\|_r+\|\ov{f}-\ov{g}\|_r)\notag\\
		&\le (1+M\tau)\lVert f-g\rVert_r.
	\end{align}
	where $M$ depends on $r$ and $\|f\|_r+\|g\|_r$ and the proof is complete.
\end{proof}

\subsection{Proof of Theorem \ref{1orderth1}}
\begin{proof}
In spirit of \eqref{firstsch}, it suffices to show
\[\|u(t_n)-u^n\|_r\le C\tau.\]
Combining the local error estimate in Lemma \ref{localerror} and stability inequality \eqref{stab1}, we are led to
\begin{align*}
			\lVert u(t_{n+1})-u^{n+1} \rVert_r&\leq \lVert u(t_{n+1})-\Psi_1^{\tau}(u(t_n)) \rVert_r+\lVert \Psi_1^{\tau}(u(t_n))-\Psi_1^{\tau}(u^{n}) \rVert_r\\
			&\leq L \tau^{2}+\lVert \Psi_1^{\tau}(u(t_n))-\Psi_1^{\tau}(u^{n}) \rVert_r\\
			&\leq  L\tau^{2}+{\rm e}^{\tau M}\lVert u(t_{n})-u^{n} \rVert_r,
		\end{align*}
where $L$ depends on $\|u\|_{L^\infty(0, T; H^{r+p(r)})}$ (or equivalently $\|z\|_{L^\infty(0, T; H^{r+p(r)})}+\|z_t\|_{L^\infty(0, T; H^{r+p(r)-2})}$), and $M$ depends on $\|u(t_n)\|_r$ and $\|u^n\|_r$. Then the assertion follows by a standard induction argument \cite{Ostermann2019,Ostermann2018,ostermann2019two}.
\end{proof}

\section{Numerical experiments}\label{num}
In this section, we present some numerical experiments of the newly proposed first-order LREI $\Psi_1^{\tau}$ to justify our theoretical convergence results. The numerical investigations of convergence of the first-order LREIs in \cite{li2022lowregularity,ostermann2019two} will be provided as comparisons. Furthermore, the Fourier pseudospectral method is used for spatial discretization so that each iteration can be calculated by FFT via $O(M{\rm log}M)$ operations, where $M$ represents the number of grid points in space. We choose the spatial mesh size $\Delta x=1/2^6$ for soliton solutions and $\Delta x=\pi/2^{15}$ for rough solutions. Moreover, we define the error in $H^r$ as $\|z(t_n)-z^n\|_r +\|\pa_t z(t_n)-z_t^n\|_{r-2}$.

\subsection{Soliton solutions}
In the first experiment, we numerically verify the temporal convergence of the first-order LREI $\Psi_1^{\tau}$ \eqref{firstsch} for the soliton solution \cite{manoranjan1984numerical} of \eqref{GB}
\begin{align}\label{solitonsolu}
	z(x,t)=-A {\rm sech}^2[(\omega/2)(x-vt+\zeta_0)],
\end{align}
where $\zeta_0\in\mathbb{R}$, $0<\omega\leq 1$, and the relations among the amplitude $A$, velocity $v$ and frequency $\omega$ are given as follows
\begin{align*}
	A=3\omega^2/2, \quad v=\pm(1-\omega^2)^{1/2}.
\end{align*}
It can be clearly observed that \eqref{solitonsolu} decays exponentially at far field, which enables us to impose the periodic boundary conditions on a bounded domain $[-x_0, x_0]$ when $x_0$ is chosen large enough. Fig. \ref{solitonsol} shows the error in $H^2$, for the numerical solution obtained by \eqref{firstsch} at $t_n=1$, where $x_0=80$. From Fig. \ref{solitonsol} we can see that the scheme converges at the first order in time.

\begin{figure}[htbp]
	\center
	\hspace{-130pt}
	\subfigure{
		\begin{minipage}[c]{0.4\linewidth}
			\centering
			\includegraphics[height=5.2cm, width=2\linewidth]{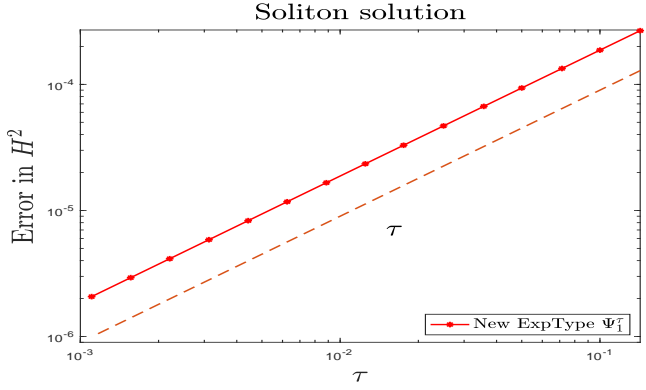}
		\end{minipage}
	}
	\caption{Linear convergence of the LREI scheme $\Psi_1^{\tau}$ for the soliton solution with $\zeta_0=0$, $\omega=1/2$,  $v=\sqrt{3}/2$. Here we select the spatial mesh size as $\Delta x=1/2^6$. } \label{solitonsol}
\end{figure}

\subsection{Rough solutions}
In the second experiment, we apply the first-order LREI $\Psi_1^{\tau}$ \eqref{firstsch} to \eqref{GB} under nonsmooth initial data. We numerically compare the results with those of the first-order LREIs in \cite{li2022lowregularity} and \cite{ostermann2019two}.

Following the construction method of the initial data with desired regularity in \cite{ostermann2019two}, we choose the spatial mesh size $\Delta x=2\pi /M$ with $M=2^{16}$ and the grid points $x_j=-\pi+j\Delta x$, $0\leq j<M$. Moreover, a uniformly distributed random vector rand($1,M$) can be taken in the computer, which is denoted by $Z=(z_0, \dots, z_{M-1})=\mathrm{rand}(1, M)$. Then we define the inverse derivative operator $\lv\partial_{x, M}\rv^{-\theta}$ as a truncation of the operator $\lv\partial_{x}\rv^{-\theta}$ \eqref{def1}, which maps a function $f\in L^2(\mathbb{T})$ to $H^{\theta}(\mathbb{T})$
\[\lv\partial_{x, M}\rv^{-\theta}f=\sum\limits_{k=-M/2, k\neq 0}^{M/2-1}\lv k\rv^{-\theta} \widehat{f}_k {\rm e}^{ikx}, \quad \theta\in\mathbb{R}.\]
Then we define
\[Z_0=\frac{Z_1+c*\|Z_1\|_{\infty}}{\|Z_1+c*\|Z_1\|_{\infty}\|_{}},\]
where $c:={\rm rand}(1)$ is a random number and $Z_1:=\lv\partial_{x, M}\rv^{-\theta} Z$, $x\in \mathbb{T}$. Finally, we get $Z_0\in H^{\theta}(\mathbb{T})$. For instance,  Fig. \ref{initia} displays the initial data obtained as above for $\theta=2$ and $\theta=2.5$, respectively.

\begin{figure}[htbp]
	\center
	\hspace{-47pt}
	\subfigure{
		\begin{minipage}[c]{0.4\linewidth}
			\centering
			\includegraphics[height=5.2cm, width=1.4\linewidth]{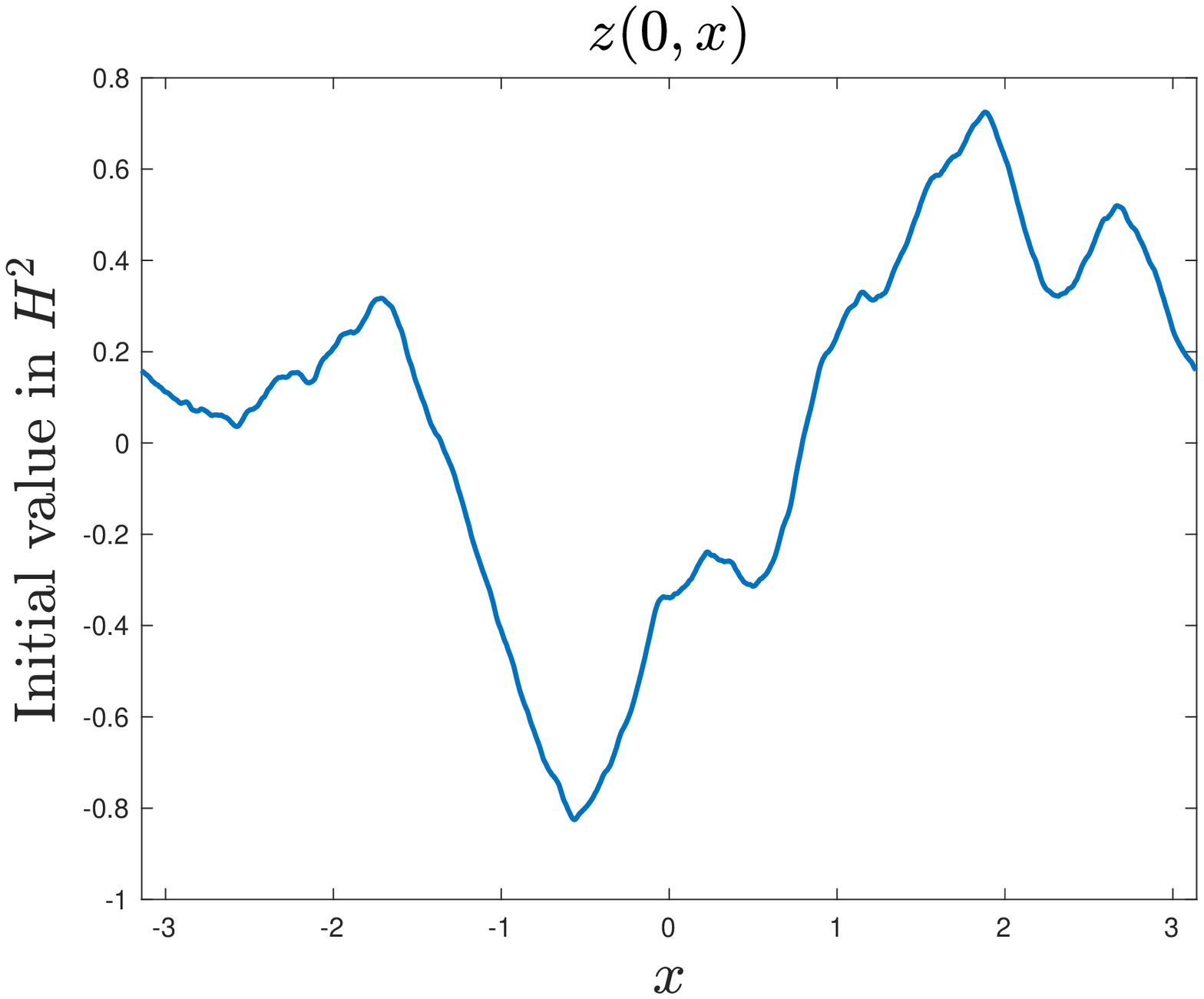}
		\end{minipage}
	}\hspace{27.7pt}
	\subfigure{
		\begin{minipage}[c]{0.4\linewidth}
			\centering
			\includegraphics[height=5.2cm, width=1.4\linewidth]{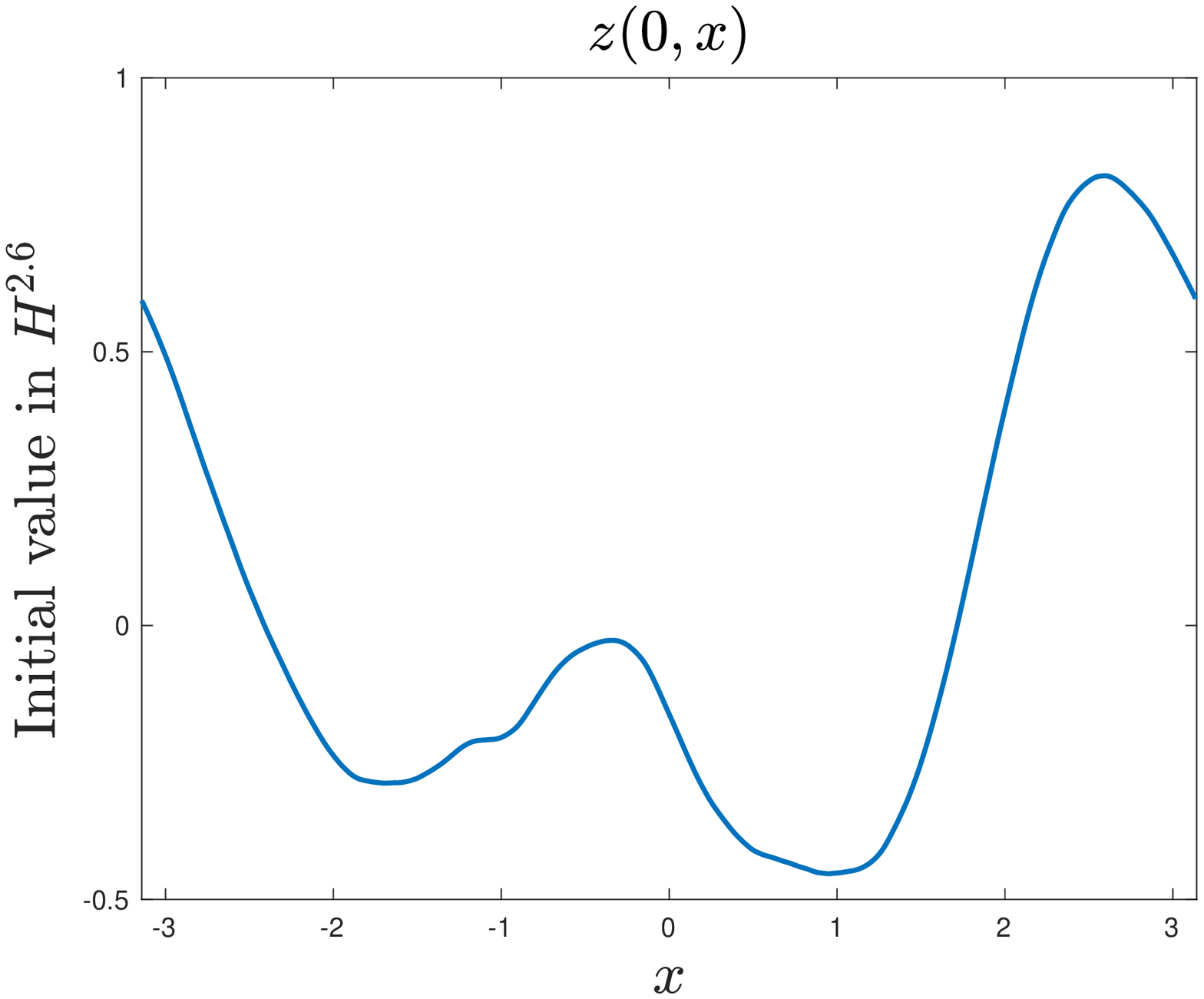}
		\end{minipage}
	}
	\caption{Left: initial value $z_0$ $\in$ $H^{2}$ with $\theta=2$. Right: initial value $z(0, x)$ $\in$ $H^{2.5}$ with $\theta=2.5$.} \label{initia}
\end{figure}

Figs. \ref{5252}-\ref{3254} show the errors of the scheme $\Psi_1^{\tau}$  and those in \cite{li2022lowregularity,ostermann2019two} in $H^r$ with various $r$ at the final time $t_n=T=1$ for different rough initial data, where the reference solution is obtained  by the first-order LREI $\Psi_1^{\tau}$ \eqref{firstsch} with a tiny time step $\tau=3\times10^{-5}$. Specifically, Figs. \ref{5252}--\ref{3254} display the temporal errors of the three schemes when the given initial data has additional order of regularity $0$, $1/4$, $1/2$, $2/3$ and $1$, respectively.
From the numerical results shown in Figs. \ref{5252} to \ref{3254}, it can be clearly observed that:
\begin{enumerate}[ (1)]
		\item The newly developed first-order LREI scheme in \eqref{firstsch} has first-order convergence in all cases, which demonstrates the theoretical results presented in Theorem \ref{1orderth1}.

		\item Compared to the other two schemes in \cite{li2022lowregularity,ostermann2019two}, the newly proposed scheme behaves most regularly and the oscillations are the weakest while the method in \cite{ostermann2019two} behaves most irregularly and might suffer an order reduction (cf. Fig. \ref{5252}). Furthermore, the method presented in this paper is the most accurate when the time step is small enough. This shows the superiority of the newly proposed method \eqref{sch0}.
	\end{enumerate}

\begin{figure}[htbp]
	\center
	\hspace{-47pt}
	\subfigure{
		\begin{minipage}[c]{0.4\linewidth}
			\centering
			\includegraphics[height=5.2cm, width=1.4\linewidth]{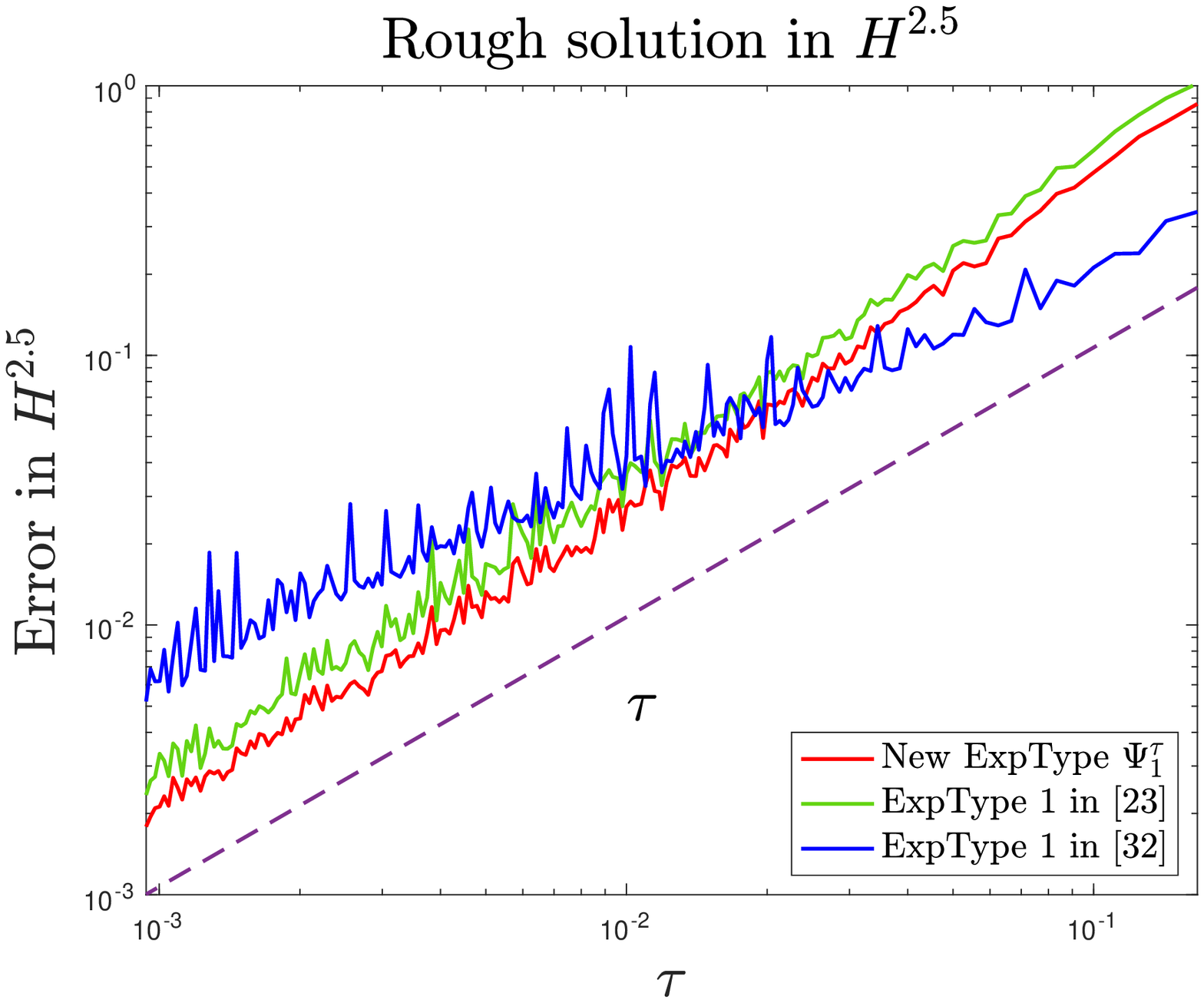}
		\end{minipage}
	}\hspace{27.7pt}
	\subfigure{
		\begin{minipage}[c]{0.4\linewidth}
			\centering
			\includegraphics[height=5.2cm, width=1.4\linewidth]{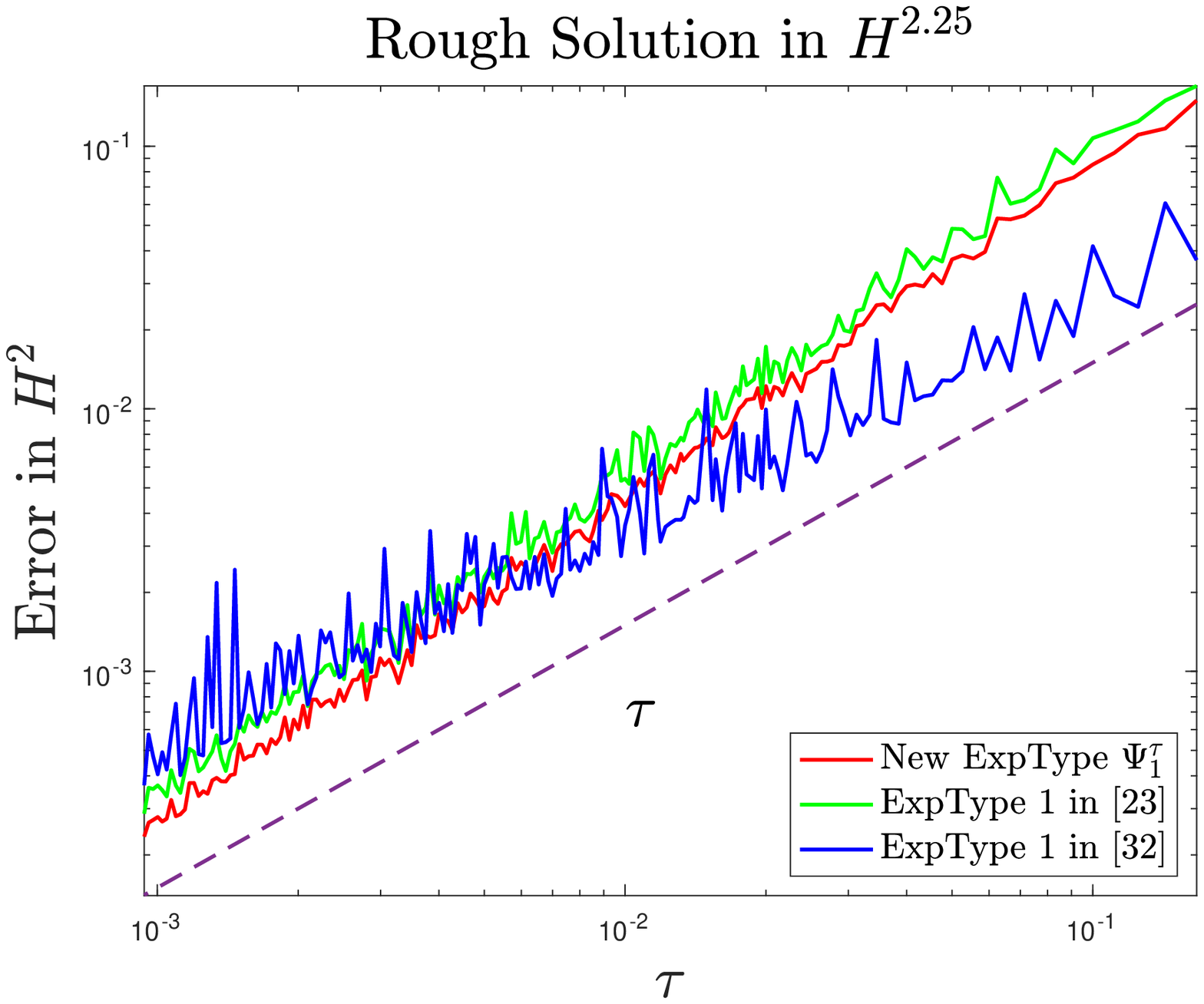}
		\end{minipage}
	}
\vspace{-2mm}
	\caption{Numerical errors in $H^{2.5}$ (left) and $H^{2}$ (right) of three first-order schemes at the final time $T=1$ with rough solution in $H^{2.5}$ and $H^{2.25}$, respectively.} \label{5252}
\end{figure}

\begin{figure}[htbp]
	\center
	\hspace{-47pt}
	\subfigure{
		\begin{minipage}[c]{0.4\linewidth}
			\centering
			\includegraphics[height=5.2cm, width=1.4\linewidth]{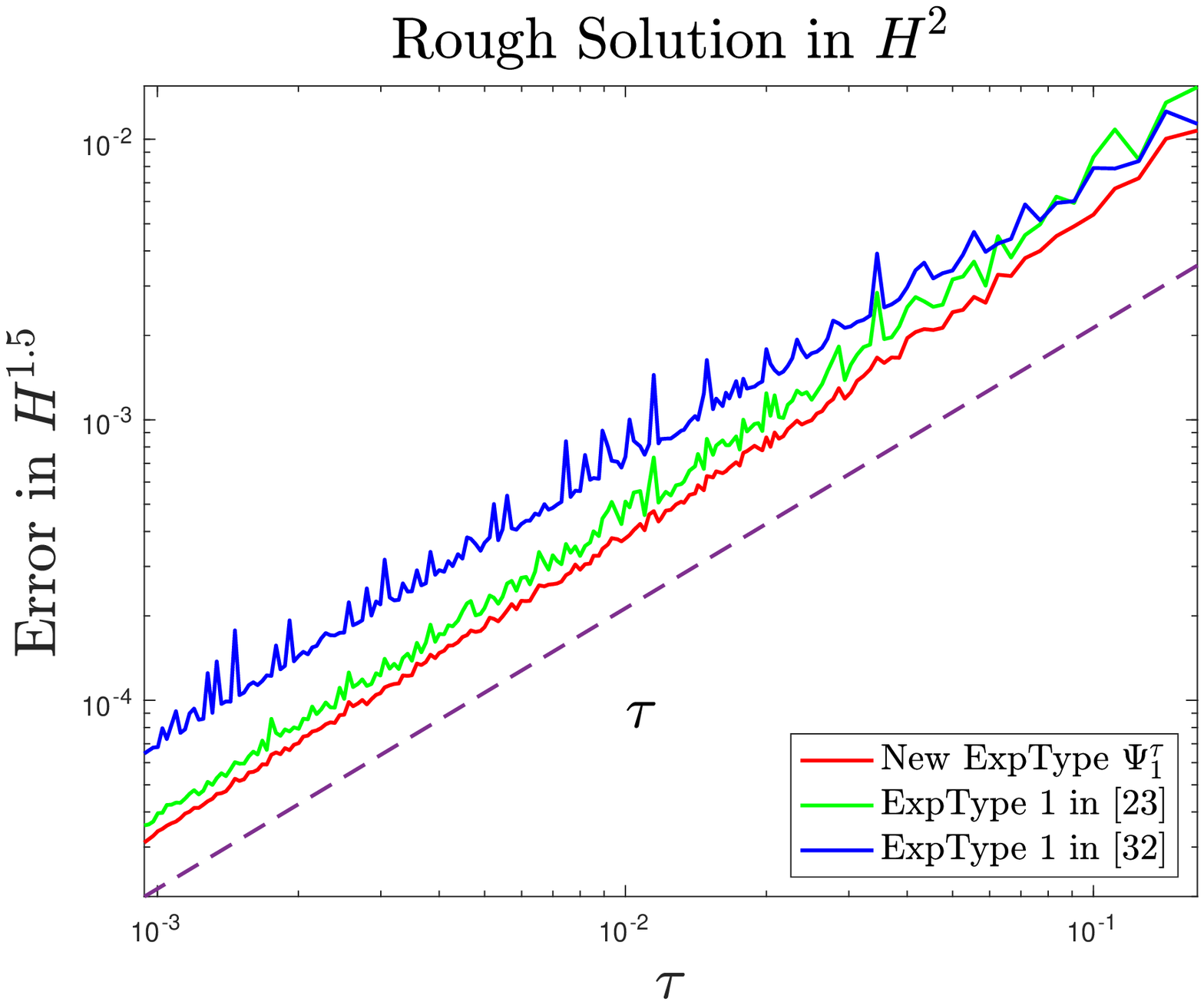}
		\end{minipage}
	}\hspace{27.7pt}
	\subfigure{
		\begin{minipage}[c]{0.4\linewidth}
			\centering
			\includegraphics[height=5.2cm, width=1.4\linewidth]{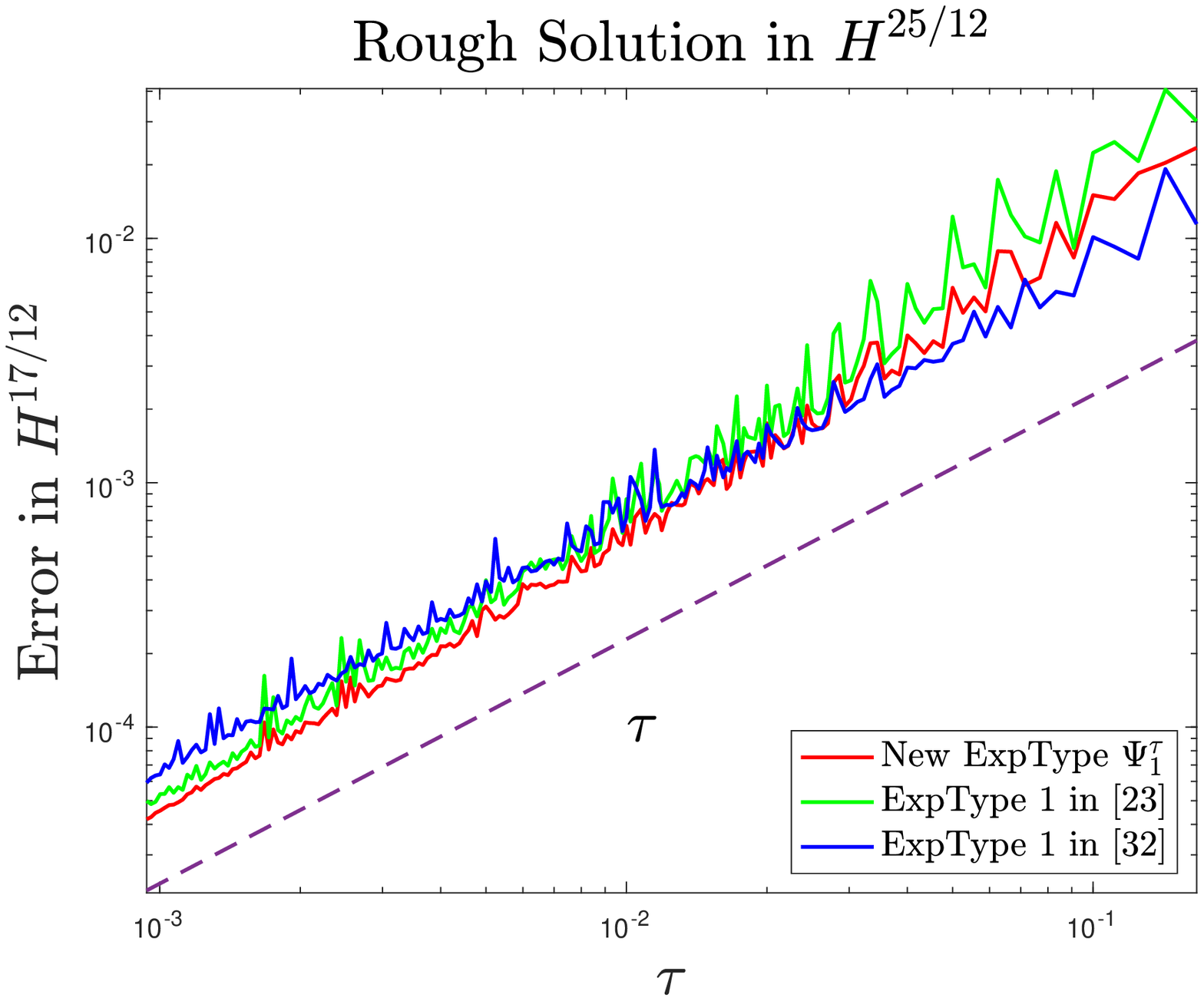}
		\end{minipage}
	}
\vspace{-2mm}
	\caption{Numerical errors in $H^{1.5}$ and $H^{17/12}$ of three first-order schemes at the final time $T=1$ with rough solution in $H^{2}$ and $H^{25/12}$, respectively.} \label{322}
\end{figure}

\begin{figure}[htbp]
	\center
	\hspace{-47pt}
	\subfigure{
		\begin{minipage}[c]{0.4\linewidth}
			\centering
			\includegraphics[height=5.2cm, width=1.4\linewidth]{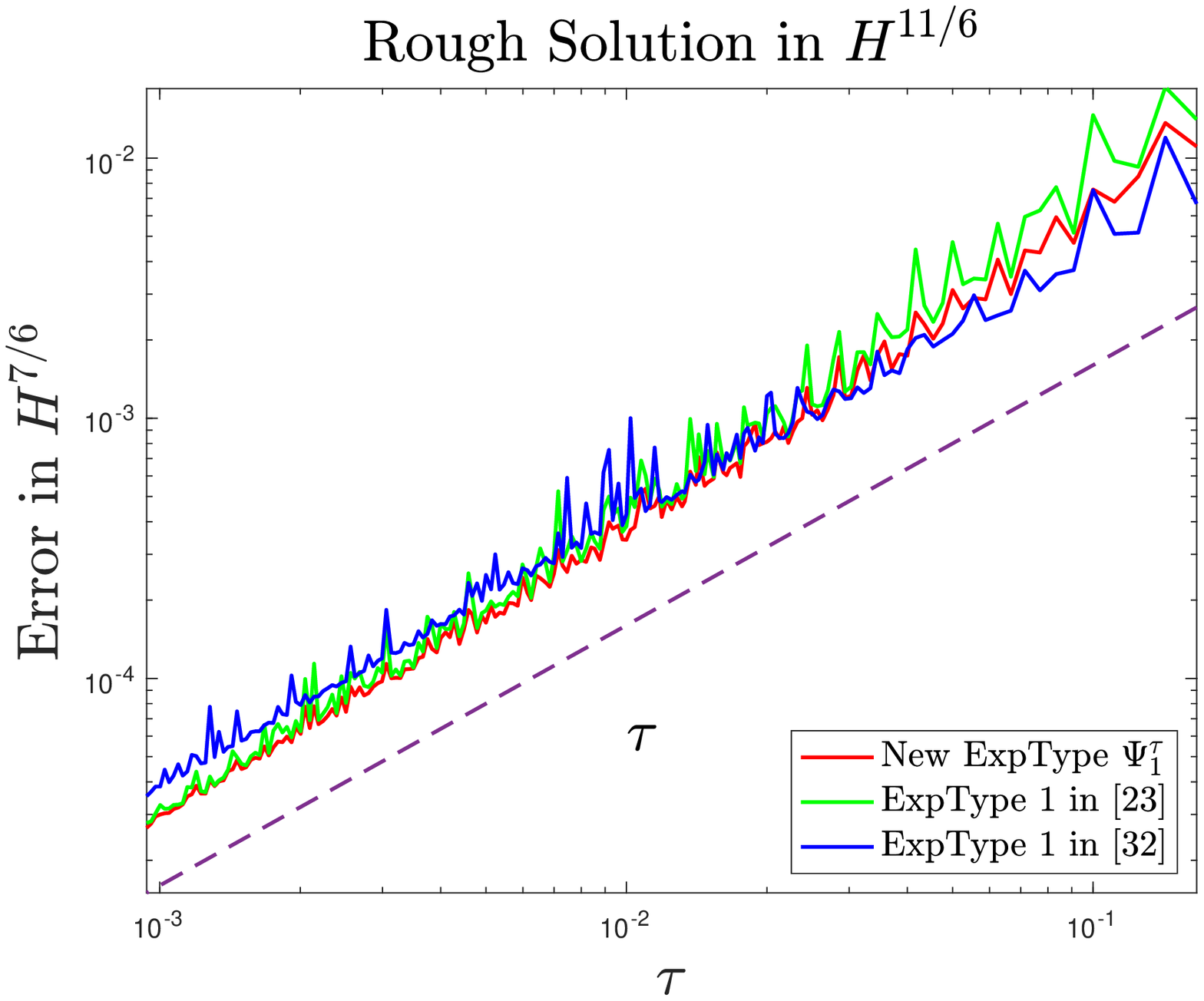}
		\end{minipage}
	}\hspace{27.7pt}
	\subfigure{
		\begin{minipage}[c]{0.4\linewidth}
			\centering
			\includegraphics[height=5.2cm, width=1.4\linewidth]{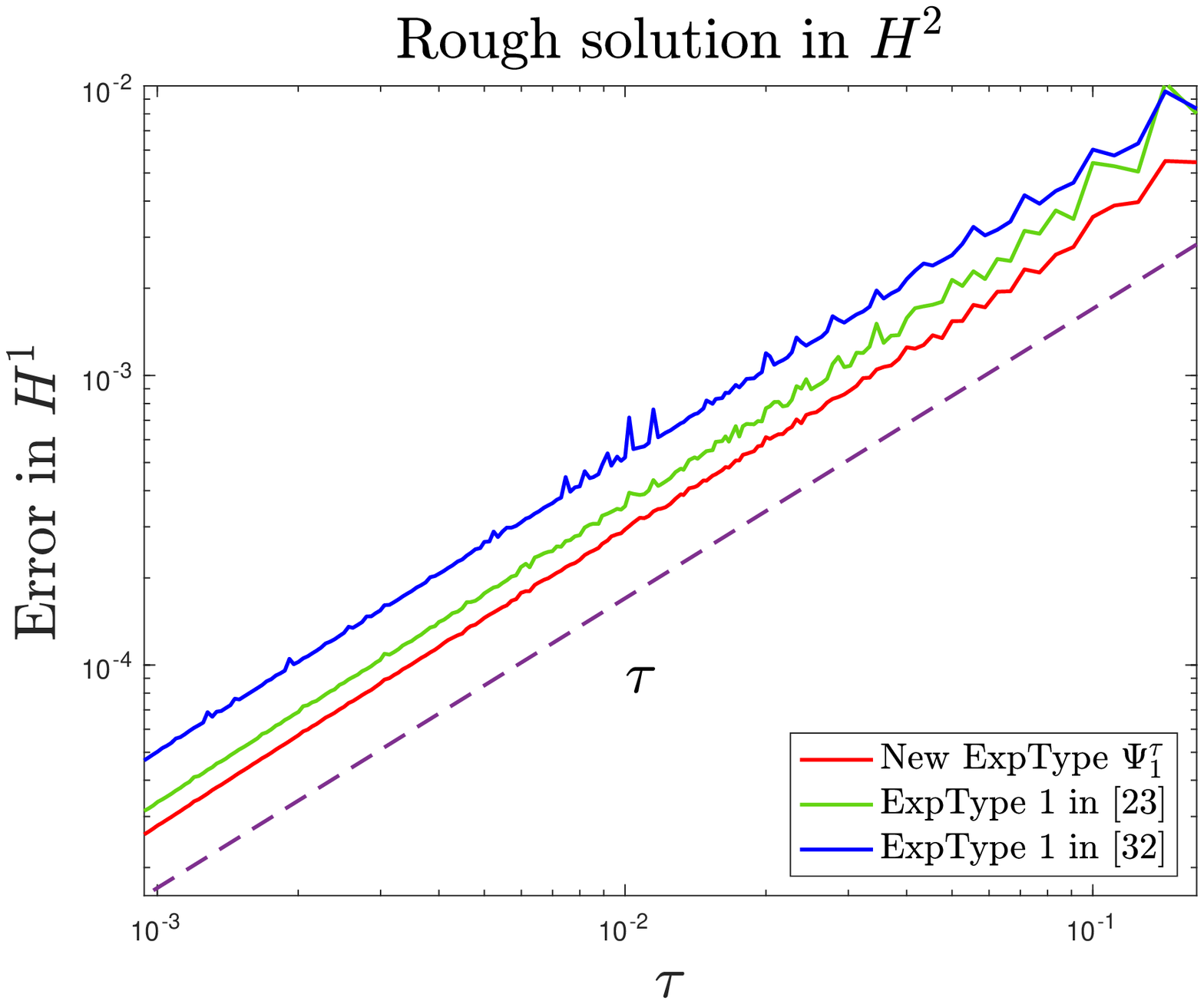}
		\end{minipage}
	}
\vspace{-2mm}
	\caption{Numerical errors in $H^{7/6}$ and $H^1$ of the three first-order schemes at the final time $T=1$ with rough solution in $H^{11/6}$ and $H^{2}$, respectively.} \label{3254}
\end{figure}

\section{Conclusions}\label{conclu}
 In this work, we developed a new first-order low regularity exponential-type integrator for the ``good" Boussinesq equation with rough initial data. The method is based on a twisted variable and the phase space analysis of the nonlinear dynamics. By applying the Kato-Ponce inequalities, the Hardy-Littlewood-Sobolev type inequality and Sobolev embedding theorem, we established the linear convergence in $H^r$ with solutions in $H^{r+p(r)}$ for $r\ge 1$, where $p(r)$ is non-increasing with respect to $r$. Particularly, the first-order accuracy can be achieved in $H^r$ for solutions in $H^r$ when $r\ge 5/2$. This is the lowest regularity requirement of the existing methods for the GB equation so far. The analytical result is supported by extensive numerical experiments.

\begin{acknowledgements}
This work was supported by the NSFC 12201342.
\end{acknowledgements}
\textbf{Data Availability} Data sharing not applicable to this article as no datasets were generated or analyzed during the current study.
%
\section*{Declarations}
\textbf{Conflict of interest} The author declares no conflict of interest.


\begin{thebibliography}{47}
	%
	%
	
	
	
	\bibitem{adams2003sobolev}
	Adams, R.~A., Fournier, J.~J.: Sobolev Spaces. Elsevier, New York (2003)
	
\bibitem{ambrosio2015periodic}
Ambrosio, V.: Periodic solutions for a pseudo-relativistic Schr{\"o}dinger equation.
Nonlinear Anal. \textbf{120}, 262--284 (2015)

	\bibitem{baumstark2018uniformly}
	Baumstark, S., Faou, E., Schratz, K.: Uniformly accurate exponential-type integrators for Klein-Gordon
	equations with asymptotic convergence to the classical NLS splitting.
	Math. Comput. \textbf{87} (311), 1227--1254 (2018)

\bibitem{benyi}
 B\'{e}nyi, \'{A}.,  Oh, T.: The Sobolev inequality on the torus revisited.
 Publicationes Mathematicae Debrecen. \textbf{83} (3), 359 (2013)
	
	\bibitem{bourgain2014endpoint}
Bourgain, J., Li, D.: On an endpoint Kato-Ponce inequality.
Differ. Integral Equa. \textbf{27} (11/12), 1037--1072 (2014)
	
	\bibitem{Boussinesq1872}
	Boussinesq,  J.: Th\'{e}orie des ondes et des remous qui se propagent le long d'un canal rectangulaire horizontal, en communiquant au liquide contenu dans ce canal
	des vitesses sensiblement pareilles de la surface au fond.
J. Math. Pures Appl. \textbf{7}, 55--108 (1872)
	
	\bibitem{bratsos2007second}
Bratsos, A.~G.: A second order numerical scheme for the solution of the
	one-dimensional Boussinesq equation.
 Numer. Algorithms \textbf{46} (1), 45--58 (2007)
	
	\bibitem{cheng2015fourier}
	Cheng, K., Feng, W., Gottlieb, W., Wang, C.: A {F}ourier pseudospectral method for the ``good" Boussinesq equation with second-order temporal accuracy.
 Numer. Methods Partial Differ. Equ.
	\textbf{31} (1), 202--224 (2015)
	
	\bibitem{el2003numerical}
	El-Zoheiry, H.: Numerical investigation for the solitary waves interaction of the
	``good" Boussinesq equation.
 Appl. Numer. Math. \textbf{45} (2-3), 161--173 (2003)
	
	\bibitem{farah2010periodic}
	Farah, L., Scialom, M.:
	 On the periodic ``good" Boussinesq equation.
	Proc. Amer. Math. Soc. \textbf{138}	(3), 953--964 (2010)
	
	\bibitem{farah2009local}
	 Farah, L.~G.: Local solutions in Sobolev spaces with negative indices for the
	``good" Boussinesq equation.  Commun. Partial Differ. Equ. \textbf{34}
	(1), 52--73 (2009)
	
	\bibitem{de1991pseudospectral}
	Frutos, J.~De, Ortega, T., Sanz-Serna, J.: Pseudospectral method for the ``good" Boussinesq equation. Math. Comput. \textbf{57} (195), 109--122 (1991)
	
	
	\bibitem{hofmanova2017exponential}
	Hofmanov{\'a}, M., Schratz, K.: An exponential-type integrator for the KdV equation. Numer. Math. \textbf{136} (4), 1117--1137 (2017)
	
	\bibitem{johnson1997modern}
	 Johnson, R.~S.: A modern introduction to the mathematical theory of water
	waves. Number~19. Cambridge University Press, Cambridge (1997)

	
	\bibitem{Kato1988CommutatorEA}
	Kato, T., Ponce, G.: Commutator estimates and the Euler and Navier-Stokes equations. Commun. Pure Appl. Math. \textbf{41}, 891--907 (1988)
	
	
	\bibitem{kirby1996nonlinear}
	 Kirby, J.~T.: Nonlinear, dispersive long waves in water of variable depth.
	Technical report, Delaware Univ. Newark Center Appl. Coastal
	Research (1996)
	
	\bibitem{kishimoto2013sharp}
	Kishimoto, N.: Sharp local well-posedness for the ``good" Boussinesq equation.
 J. Differ. Equ. \textbf{254} (6), 2393--2433 (2013)
	
	\bibitem{Ostermann2019}
Knoller, M., Ostermann, A., Schratz, K.:
 A Fourier integrator for the cubic nonlinear
	Schr\"odinger equation with rough initial data.
 SIAM J. Numer. Anal. \textbf{57},  1967--1986 (2019)
	
	
	\bibitem{lambert1987soliton}
Lambert, F., Musette, M., Kesteloot, E.: Soliton resonances for the good Boussinesq equation. Inverse Probl. \textbf{3} (2), 275 (1987)

\bibitem{li2021fully}
Li, B., Wu, Y.: A fully discrete low-regularity integrator for the 1d periodic cubic
	nonlinear Schr{\"o}dinger equation. Numer. Math. \textbf{149} (1), 151--183 (2021)
	
\bibitem{LiWu22}
Li, B., Wu, Y.: An unfiltered low-regularity integrator for the KdV equation with solutions below $H^1$. arXiv:2206.09320 (2022)
	
	\bibitem{li2022lowregularity}
Li, H.,  Su, C.: Low regularity exponential-type integrators for the ``good" Boussinesq equation, to appear in IMA J. Numer. Anal. (2022)

\bibitem{li2019kato}
	Li, L.: On Kato--Ponce and fractional Leibniz.
Rev. Mat. Iberoam. \textbf{35} (1), 23--100 (2019)
	
	\bibitem{manoranjan1984numerical}
	Manoranjan, V., Mitchell, A.,  Morris, J. L.: Numerical solutions of the good Boussinesq equation. SIAM J. Sci. Stat. Comput. \textbf{5} (4), 946--957 (1984)
	
	\bibitem{manoranjan1988soliton}
	Manoranjan, V., Ortega, T., Sanz-Serna, J.:
	Soliton and antisoliton interactions in the ``good" Boussinesq
	equation.  J. Math. Phys. \textbf{29} (9), 1964--1968 (1988)
	
\bibitem{maz2002bourgain}
Maz'ya, V., Shaposhnikova, T.: On the Bourgain, Brezis, and Mironescu theorem concerning limiting embeddings of fractional Sobolev spaces.
J. Funct. Anal. \textbf{195} (2), 230--238 (2002)
	
	\bibitem{oh2013improved}
	Oh, S., Stefanov, A.: Improved local well-posedness for the periodic ``good" Boussinesq equation. J. Differ. Equ. \textbf{254} (10), 4047--4065 (2013)
	
	\bibitem{ortega1990nonlinear}
	Ortega, T., Sanz-Serna, J.: Nonlinear stability and convergence of finite-difference methods for the ``good" Boussinesq equation.
Numer. Math. \textbf{58} (1),  215--229 (1990)
	
	\bibitem{Ostermann2021}
	Ostermann, A., Rousset, F., Schratz, K.: Error estimates of a Fourier integrator for the cubic Schr\"odinger equation at low regularity.
 Found. Comput. Math. \textbf{21} (3), 725--765 (2021)
	
	\bibitem{Ostermann2018}
	Ostermann,  A., Schratz, K.:
	 Low regularity exponential-type integrators for semilinear
	Schr\"odinger equations.
	 Found. Comput. Math. \textbf{18}, 731--755 (2018)
	
	\bibitem{ostermann2019two}
	Ostermann, A., Su, C.:	Two exponential-type integrators for the ``good" Boussinesq equation.
 Numer. Math. \textbf{143} (3), 683--712 (2019)
	
	\bibitem{ostermann2020lawson}
	Ostermann, A., Su, C.:
 A lawson-type exponential integrator for the Korteweg--de Vries
	equation. IMA J. Numer. Anal. \textbf{40} (4), 2399--2414 (2020)
	
	\bibitem{ostermann2022fully}
Ostermann, A., Wu, Y., Yao, F.:
 A second-order low-regularity integrator for the nonlinear Schr\"odinger equation.
Adv. Cont. Discr. Mod. \textbf{91} (1), 1--14 (2022)
	
	
	\bibitem{rousset2021general}
	Rousset, F., Schratz, K.: A general framework of low regularity integrators.
	SIAM J. Numer. Anal. \textbf{59} (3), 	1735--1768 (2021)
	
	\bibitem{schratz2021low}
	Schratz, K., Wang, Y., Zhao, X.:
 Low-regularity integrators for nonlinear Dirac equations.
 Math. Comput. \textbf{90} (327), 189--214 (2021)

	\bibitem{stein2016introduction}
	Stein, E. M., Weiss, G.: Introduction to Fourier Analysis on Euclidean Spaces.
	Princeton university press (2016)
	
	\bibitem{su2020deuflhard}
Su, C., Yao, W.: A Deuflhard-type exponential integrator fourier pseudo-spectral
	method for the ``good" Boussinesq equation.
J. Sci. Comput. \textbf{83} (1), 1--19 (2020)

\bibitem{TTao}
	 Tao, T.: Nonlinear Dispersive Equations. Local and Global Analysis.
	Amer. Math. Soc., Providence RI (2006)

\bibitem{tatlock2018assessment}
Tatlock, B., Briganti, R., Musumeci, R.~E., Brocchini, M.:
 An assessment of the roller approach for wave breaking in a hybrid
	finite-volume finite-difference boussinesq-type model for the surf-zone.
 Appl. Ocean Res. \textbf{73}, 160--178 (2018)
	
	\bibitem{varlamov2001eigenfunction}
Varlamov, V.: Eigenfunction expansion method and the long-time asymptotics for the damped Boussinesq equation. Discrete Contin. Dyn. Syst. \textbf{7} (4),
	675-702 (2001)

\bibitem{wang2022symmetric}
Wang, Y., Zhao, X.: A symmetric low-regularity integrator for nonlinear Klein-Gordon equation. Math. Comput. \textbf{91} (337), 2215--2245 (2022)
	
	
	\bibitem{wang2013well}
Wang, H., Esfahani, A.: Well-posedness for the Cauchy problem associated to a periodic
	Boussinesq equation. Nonlinear Anal. \textbf{89}, 267--275 (2013)
	
	\bibitem{wu2020first}
Wu, Y., Yao, F.: A first-order Fourier integrator for the nonlinear Schr{\"o}dinger
	equation on $\mathbb{T}$ without loss of regularity.
 Math. Comput. \textbf{91} (335), 1213--1235 (2022)
	
	\bibitem{wu2019optimal}
Wu, Y., Zhao, X.: Optimal convergence of a second order low-regularity integrator for
	the KdV equation. IMA J. Numer. Anal. doi.org/10.1093/imanum/drab054 (2021)
	
	\bibitem{wu2022embedded}
Wu, Y., Zhao, X.: Embedded exponential-type low-regularity integrators for KdV equation under rough data. BIT Numer. Math. \textbf{62} (3), 1049--1090 (2022)
	
	
	\bibitem{zhang2017second}
Zhang, C., Wang, H., Huang, J., Wang, C., Yue, X.:
	A second order operator splitting numerical scheme for the ``good"
	Boussinesq equation. Appl. Numer. Math. \textbf{119}, 179--193 (2017)
	
\end{thebibliography}


%
%

\end{document}